\documentclass[11pt,reqno,oneside]{amsart}
\usepackage{mathdots}
\usepackage{amsmath}
\usepackage{amssymb}
\usepackage{amsthm}
\usepackage{graphics}
\usepackage{eucal}
\usepackage{bm}
\usepackage[colorlinks]{hyperref}
\usepackage[capitalize,nameinlink,noabbrev]{cleveref}
\usepackage{mathrsfs}
\usepackage[normalem]{ulem}
\usepackage{color}
\usepackage[dvipsnames]{xcolor}
\usepackage[colorinlistoftodos,prependcaption,textsize=tiny]{todonotes}
\usepackage{mathtools}
\usepackage{geometry}

\newtheorem{theorem}{Theorem}[section]
\newtheorem{proposition}{Proposition}[section]
\newtheorem{lemma}{Lemma}[section]
\newtheorem{corollary}{Corollary}[section]

\theoremstyle{remark}
\newtheorem{remark}{Remark}[section]
\theoremstyle{definition}
\newtheorem{definition}{Definition}[section]

\numberwithin{equation}{section}

\newcommand{\vv}[1]{{\mathbf{#1}}}
\newcommand{\ep}{\varepsilon}

\newcommand{\be}{\mathbf{e}}

\newcommand{\diag}{\operatorname{diag}}

\newcommand{\GL}{\operatorname{GL}}

\newcommand{\tDelta}{\Delta^{\!1}}

\newcommand{\bx}{\mathbf{x}}
\newcommand{\bU}{\mathbf{U}}

\newcommand{\tF}{\tilde{F}}

\newcommand{\bp}{\mathbf{p}}

\newcommand{\brho}{\boldsymbol{\rho}}
\newcommand{\br}{\mathbf{r}}
\newcommand{\bPsi}{\boldsymbol{\Psi}}

\newcommand{\fs}{\mathfrak{s}}

\newcommand{\bv}{\mathbf{v}}
\newcommand{\bw}{\mathbf{w}}

\newcommand{\ve}{\varepsilon}
\newcommand{\bve}{\boldsymbol{\varepsilon}}

\newcommand{\cS}{\mathcal{S}}
\newcommand{\ba}{\mathbf{a}}
\newcommand{\bb}{\mathbf{b}}

\newcommand{\Z}{\mathbb{Z}}
\newcommand{\R}{\mathbb{R}}
\newcommand{\N}{\mathbb{N}}
\newcommand{\fff}{\mathit{f}}

\newcommand{\bT}{\mathbf{T}}
\newcommand{\bK}{\mathbf{K}}

\newcommand{\cM}{\mathcal{M}}

\usepackage{mathtools}
\usepackage{tikz-cd}
\usepackage{graphicx}
\newcommand{\Rp}{\mathbb{R}_+}
\newcommand{\cSM}{{\mathcal{S}}^\times}
\newtheorem{thkh}{Theorem}

\newtheorem{thga}{Theorem}

\newtheorem{thkw}{ Theorem KW}

\begin{document}

\title[Weighted and multiplicative Khintchine theorems]{Weighted Diophantine approximation on manifolds}

\author{Victor Beresnevich}\address{\textbf{Victor Beresnevich} \\
Department of Mathematics, University of York, UK }
\email{victor.beresnevich@york.ac.uk}
\author{Shreyasi Datta}
\address{\textbf{Shreyasi Datta} \\
Theoretical Statistics and Mathematics Unit, Indian Statistical Institute, Bangalore, India}
\email{shreyasi@isibang.ac.in, shreyasi1992datta@gmail.com}
\author{Lei Yang}
\address{\textbf{Lei Yang} \\
Department of Mathematics, National University of Singapore, Singapore}
\email{lei.yang@nus.edu.sg}

\date{}

\maketitle

{\centering \textit{In memory of Vasili Bernik}\par}
\begin{abstract}
We establish a weighted simultaneous Khintchine-type theorem, both convergence and divergence, for all nondegenerate manifolds, which answers a problem posed in [Math. Ann., 337(4):769–796, 2007]. This extends the main results of [Acta
Math., 231:1–30, 2023] and  [Ann. of Math.
(2), 175(1):187–235, 2012] in the weighted set-up. As a by-product of our method, we also obtain a multiplicative Khintchine-type convergence theorem for all nondegenerate manifolds, which is a simultaneous analogue of the celebrated result of Bernik, Kleinbock, and Margulis for dual approximation. 
\end{abstract}

\section{Introduction}

One classical theme in the theory of Diophantine approximation is to quantify how well rational points 
$$
\bp/q=(p_1/q,\dots,p_n/q)\,,
$$ 
with $q\in\N$ and $\bp\in\Z^n$, approximate real points $\bx=(x_1,\dots,x_n)\in\R^n$.
In the standard setting, the error of approximations is measured using a fixed norm on $\R^n$, usually the sup-norm 
$$
\|\bx\|=\max_{1\le i\le n}|x_i|\,.
$$
Thus one seeks the same rate of approximation in each coordinate, imposing identical bounds on the quantities $|x_i-p_i/q|$.

In this paper we consider two more general frameworks: \emph{weighted} and \emph{multiplicative}, both of which can be seen as important special cases of the more general theory of Diophantine approximation with star bodies \cite{MR2264601}. In the weighted case, different coordinates may be approximated at different rates; in the multiplicative case, one specifies a bound on the product
$\prod_{i=1}^n |x_i - p_i/q|$.
We now introduce the relevant notation and terminology.

Let $\Rp$ denote the set of positive real numbers, and let $\psi,\psi_1,\dots,\psi_n:\Rp\to(0,1)$ be nonincreasing  functions, referred to as {\em approximation functions}.
Define $\mathcal{S}_n(\psi_1,\dots, \psi_n)$ as the set of $\bx\in\R^n$ such that the system
\begin{equation}\label{eq1.1}
\vert qx_i-p_i\vert<\psi_i(q)\qquad (1\leq i\leq n) 
\end{equation}
holds for infinitely many $(q,p_1,\dots,p_n)\in\N\times\Z^n$. When $\psi_1=\dots=\psi_n=\psi$ this set will be denoted by $\cS_n(\psi)$. A point $\bx\in\cS_n(\psi_1,\dots,\psi_n)$ will be called {\em $(\psi_1,\dots,\psi_n)$--approximable} ({\em $\psi$-approximable} when $\psi_1=\dots=\psi_n=\psi$).

In the multiplicative case,
define $\cSM_n(\psi)$ as  the set of $\bx\in\R^n$ such that
$$  \prod_{i=1}^n |qx_i-p_i| \ < \ \psi(q)
$$  
holds for infinitely many $(q,p_1,\dots,p_n)\in\N\times\Z^n$. The points in $\cSM_n(\psi)$ will be called 
{\em multiplicatively
$\psi$--approximable}.
It is readily verified that
\begin{equation}\label{s:004}
  \cS_n(\psi_1,\dots,\psi_n)\subset\cSM_n(\psi)\quad\text{if}\quad\psi\ge\psi_1\cdots\psi_n\,.
\end{equation}

\noindent By Minkowski's linear forms theorem,
\begin{equation}\label{sv:004}
\cS_n(\psi_1,\dots,\psi_n)=\R^n \hspace{9mm}  {\rm  if } \hspace{9mm}
q\psi_1(q)\cdots\psi_n(q)\ge 1 \quad\text{infinitely often}\,.
\end{equation}
This together with \eqref{s:004} implies that
\begin{equation}\label{sv:005}
\cSM_n(\psi)=\R^n  \hspace{9mm}  {\rm  if } \hspace{9mm} q\psi(q)\ge 1 \quad\text{infinitely often} .
\end{equation}
Littlewood's conjecture, a famous open problem dating back to the 1930s, asserts that the condition in \eqref{sv:005} can be weakened to 
$$\psi(q)\ge \varepsilon q^{-1}\qquad\text{for any $\varepsilon>0$},
$$
see \cite{EKL} and \cite{PVActa}. However, in the weighted case, 
such a weakening is generally impossible for the existence
of badly approximable points \cite{MR2581371, MR2231044}, e.g. when 
$$
\psi_i(q)=q^{-\tau_i}\,,\qquad \tau_i\ge0,\quad\tau_1+\dots+\tau_n=1\,.
$$
Here $\bx\in\R^n$ is {\em $(\psi_1,\dots,\psi_n)$-badly approximable} if $\bx\not\in\cS_n(c\psi_1,\dots,c\psi_n)$ for some $c>0$. Nevertheless, the set of $(\psi_1,\dots,\psi_n)$-badly approximable points is well known to have zero measure \cite{BadNull}.

In this paper we will be concerned with the measure of the sets $\cS_n(\psi_1,\dots,\psi_n)$ and $\cSM_n(\psi)$ 
restricted to submanifolds of $\R^n$ when $\psi_1,\dots,\psi_n,\psi$ are generic non-increasing functions. 
We begin by recalling the corresponding results for $\R^n$. Let $\mu_n$ denote Lebesgue measure on $\R^n$.

\smallskip

\begin{thkh}
\label{gs} Let   $\psi_1,\dots,\psi_n$ be non-increasing functions. Then
$$\mu_{n}\left(\cS_n(\psi_1,\dots,\psi_n)\right) =\left\{\begin{array}{ll} \mbox{\rm
Z{\scriptsize ERO}} & {\rm if} \;\;\; \sum \;  \psi_1(q)\cdots \psi_n(q)
\;\; <\infty\,,\\ &
\\[-1ex] \mbox{\rm F{\scriptsize ULL}} & {\rm if} \;\;\; \sum  \;  \psi_1(q)\cdots \psi_n(q) \;\;
 =\infty\,. \; \;
\end{array}\right.$$
\end{thkh}

\smallskip 

\begin{thga}
\label{gm} Let   $\psi$ be a non-increasing function. Then
$$\mu_{n}\left(\cSM_n(\psi)\right) =\left\{\begin{array}{ll} \mbox{\rm
Z{\scriptsize ERO}} & {\rm if} \;\;\; \sum \;  \psi(q) \ (\log
q)^{n-1} \;\; < \ \infty\,,\\ &
\\[-1ex] \mbox{\rm F{\scriptsize ULL}} & {\rm if} \;\;\; \sum  \;  \psi(q) \ (\log
q)^{n-1}  \;\;
 = \ \infty\,. \; \;
\end{array}\right.$$
\end{thga}

\noindent Here `\mbox{\rm F{\scriptsize ULL}}' means that the complement has  measure zero. Theorem~K was originally established by Khintchine  
\cite{Kh} when $\psi_1=\dots=\psi_n$, and both results as stated above were obtained by Gallagher in \cite{gal}.

Our main results extend Theorem~K and the convergence part of Theorem~G to nondegenerate submanifolds of $\R^n$ as defined in \cite{KM}. 

\smallskip

\begin{theorem}[\textsc{Weighted Khintchine for manifolds}]\label{thm: dream0}
Let $\cM$ be any nondegenerate submanifold of $\R^n$ and $\psi_1,\dots,\psi_n:\Rp\to(0,1)$ be non-increasing. Then almost all points on $\cM$ are $(\psi_1,\dots,\psi_n)$-approximable if the series
\begin{equation}\label{mainsum}
\sum_{q=1}^\infty\psi_1(q)\cdots \psi_n(q)
\end{equation}
diverges, and are not $(\psi_1,\dots,\psi_n)$-approximable if \eqref{mainsum} converges.
\end{theorem}

\medskip

\begin{theorem}[\textsc{Multiplicative convergence for manifolds}]\label{thm:mult0}Let $\cM$ be any nondegenerate submanifold of $\R^n$ and $\psi:\Rp\to(0,1)$ be non-increasing. Then almost all points on $\cM$ are not multiplicatively
$\psi$-approximable if
$$
\sum_{q=1}^\infty \psi(q) (\log q)^{n-1}<\infty.
$$
\end{theorem}

Theorem \ref{thm: dream0} is first proved for the case when the manifold is parametrized in the so-called Monge form in \S \ref{sec: div} (for divergence) and \S \ref{sec: convergence} (for convergence). The proof is then generalized to arbitrary  nondegenerate manifolds in \S \ref{sec: general}. The proof of  Theorem~\ref{thm:mult0} is presented in \S \ref{sec: multi}.

\subsection{Previous results and comparison}

We first discuss the unweighted case ($\psi_1=\cdots=\psi_n$) of Theorem~\ref{thm: dream0}. In this setting, Theorem~\ref{thm: dream0} was first established for nondegenerate $C^3$ planar curves ($n=2$) in \cite{BDV07} in the divergence case and in \cite{VV06} in the convergence case. The $C^3$ assumption was later removed in \cite{BZ10} (divergence) and \cite{Hua15} (convergence). In higher dimensions, the divergence case was proved by the first-named author in \cite{Ber12} for analytic nondegenerate manifolds and was recently extended to all nondegenerate manifolds in \cite{BD}, having previously been known only for nondegenerate curves \cite{BVVZ21}. The convergence case was settled in \cite{BY} for all nondegenerate manifolds.

Much less is known in the weighted case, which allows the approximation functions $\psi_1,\dots,\psi_n$ to be different. The story again begins with planar curves, studied in \cite{BV07}, where the authors proved both the convergence and divergence cases of Theorem~\ref{thm: dream0}: the divergence case for $C^3$ curves and the convergence case for rational quadrics. The latter was later extended to arbitrary $C^3$ nondegenerate planar curves in \cite{Bad_Lev_07}.

In higher dimensions, the general problem was posed in \cite[Problem~S2]{BV07} as follows.

\smallskip

\noindent\textbf{Problem~1:} Given a nondegenerate manifold $\cM$ and approximation functions $\psi_1,\dots,\psi_n$, find the weakest condition under which $\mathcal{S}_n(\psi_1,\dots,\psi_n)\cap\cM$ has Lebesgue measure zero.

\smallskip

Working towards this general problem, Srivastava \cite{Srivastava2025} recently established the convergence case of Theorem~\ref{thm: dream0} for a subclass of nondegenerate manifolds with a curvature condition, under the additional constraint $\psi_1=\cdots=\psi_d$ on the approximation functions. To the best of our knowledge, the divergence part of Theorem~\ref{thm: dream0} for $n>2$ and arbitrary approximation functions $\psi_1,\dots,\psi_n$ has not been previously addressed. Theorem~\ref{thm: dream0} resolves Problem~1 in full and, at the same time, generalises the results of \cite{BY} and \cite{BD} from equal to arbitrary approximation functions $\psi_1,\dots,\psi_n$.

\medskip

We now turn to the multiplicative case, concerning the set $\cSM_n(\psi)$. In this setting, progress to date has been far more limited, and even for planar curves the theory remains incomplete. Indeed, for nondegenerate planar curves, Theorem~\ref{thm:mult0} was proved in \cite{Hua15}, having been previously established for rational quadrics \cite{BV07} and under the $C^3$ assumption \cite{Bad_Lev_07}.

As with Problem~1, the general question regarding $\cSM_n(\psi)$ was posed in \cite[Problem~S2]{BV07}:

\smallskip

\noindent\textbf{Problem~2:} Given a nondegenerate manifold $\cM$ and an approximation function $\psi$, find the weakest condition under which $\cSM_n(\psi)\cap\cM$ has Lebesgue measure zero.

\smallskip

Theorem~\ref{thm:mult0} contributes towards the resolution of this problem by providing an analogue of the convergence case of Theorem~G for all nondegenerate manifolds.

A natural next step would be to prove the divergence case of Theorem~G for nondegenerate manifolds. This remains a major open problem and has so far been addressed only for lines in $\R^2$ \cite{MR4068301, MR3813593, MR4701882} and for ``vertical lines'' in higher dimensions \cite{MR4016056, MR4706444}. Even these limited cases required substantial new ideas, including the development of a Bohr sets technique \cite{MR3813593} and an effective asymptotic equidistribution method \cite{MR4701882}.

Finally, we note that our results may be viewed as complete analogues of the corresponding results for dual approximation on manifolds. In that setting, the convergence case for both weighted and multiplicative approximation was established by Bernik, Kleinbock and Margulis \cite{BKM}, while the divergence case for weighted approximation was proved in \cite{MR2989975}.

\subsection{Further remarks}

For a manifold $\cM \subset \R^n$ of dimension $d$ to be nondegenerate, it is necessary that $\cM$ be $C^l$ with ${d+l \choose l} \ge n+1$. For example, if $d=2$, this reduces to $l^2 + 3l \ge 2n$. Clearly there exist manifolds that are just $C^l$ for the smallest $l$ satisfying this inequality. The above theorems use this minimal smoothness, sufficient to meet the nondegeneracy condition.

The proof of Theorem~\ref{thm: dream0} builds on the linearization techniques developed in \cite{BY} and \cite{BD}. The new features include linearization in different coordinates with different rates, including linearization in a single direction only. For the divergence case, we use the recent \textit{ubiquity} result of Kleinbock and Wang \cite{KWang23}. To construct an appropriate ubiquitous system, we provide a version of the quantitative estimate of Bernik, Kleinbock, and Margulis \cite{BKM} that enables independent control of different partial derivatives (see \S \ref{new BKM}). At the technical level, our proof differs from the previous argument even in the \textit{unweighted} case in \cite{BD}, offering a new perspective.

Finally, we note that, while preparing this paper, we became aware of related work by Chow, Srivastava, Technau, and Yu \cite{MultCSTY} concerning Problem~2. Using a different approach initiated in \cite{SST}, they establish a multiplicative convergence result for \(C^\infty\) nondegenerate manifolds and a natural class of affine subspaces.
We emphasize that both Problem~1 and Problem~2 remain of interest and are still widely open for degenerate manifolds, such as affine subspaces and submanifolds that are nondegenerate relative to an affine subspace, as considered in \cite{MR1982150}.

\section{Set-up, notation and auxiliary statements}

\subsection{Main results through nondegenerate  maps}

We begin by recalling the definition of nondegeneracy following \cite{KM}. A map
$$
F:\bU \to \R^n
$$
is called $l$-nondegenerate at $\bx \in \bU$ if $F$ is $C^l$ on a neighbourhood of $\bx$ and there exist $n$ linearly independent partial derivatives of $F$ at $\bx$ of orders at most $l$. We say that $F$ is nondegenerate at $\bx$ if it is $l$-nondegenerate at $\bx$ for some $l$. We say that $F$ is nondegenerate in $\bU$ if it is nondegenerate at $\mu_d$-almost every $\bx \in \bU$, where $\mu_d$ denotes $d$-dimensional Lebesgue measure on $\R^d$.
Respectively, a manifold $\cM$ is called nondegenerate if almost every point on $\cM$ has a neighbourhood that can be parametrised by a nondegenerate map. Since a manifold is nondegenerate if it admits a (local) nondegenerate parametrisation, we can now restate Theorems~\ref{thm: dream0} and~\ref{thm:mult0} in terms of nondegenerate maps.

\smallskip

\begin{theorem}\label{thm: dream}
Let $F:\bU\to\R^n$ be nondegenerate, $\bU$ be an open subset of $\R^d$, $\psi_1,\dots,\psi_n:\Rp\to(0,1)$ be non-increasing. Then \begin{equation}
        \mu_d\left(F^{-1}\mathcal{S}_n(\psi_1,\dots, \psi_n)\right)=\left\{\begin{aligned}
            & 0 &\text{ if  } &\textstyle\sum\psi_1(q)\cdots \psi_n(q)<\infty\,,\\[1ex]
            & \mu_d(\bU) &\text{ if  } &\textstyle\sum \psi_1(q)\cdots \psi_n(q)=\infty\,.
        \end{aligned}\right.
    \end{equation}
\end{theorem}

\medskip

\begin{theorem}\label{thm:mult} Let $F,n,d$ be as in Theorem~\ref{thm: dream} and $\psi:\Rp\to(0,1)$ be non-increasing. Then
$$\textstyle\mu_d\big(F^{-1}\cSM_n(\psi)\big) =0\quad \text{if}\quad \sum \psi(q) (\log q)^{n-1}<\infty.
$$
\end{theorem}

\subsection{Assumptions on approximation functions}\label{sec: con on psi div}

Let $\psi_i:\N\to(0,1)$ be non-increasing for every $i=1,\dots,n$. First, we justify that, while proving Theorem~\ref{thm: dream}, without loss of generality the following chain condition may be assumed:
\begin{equation}\label{eq: psi chain}
    \psi_1(q)\leq \cdots\leq \psi_n(q)\qquad\text{for all }q.
\end{equation}

We begin by noting that there is a non-empty subset $\Pi$ of permutations $\pi$ of $(1,\dots,n)$ and disjoint strictly increasing sequences $(q^{\pi}_k)_{k\in\N}$, where $\pi\in\Pi$, 
such that for every $\pi\in\Pi$ 
\begin{equation}\label{eq2.3-}
\textstyle\N\setminus\bigcup_{\pi\in\Pi}\{q_k^\pi:k\in\N\}\qquad\text{is finite}
\end{equation}
and
\begin{equation}\label{eq2.3}
\psi_{\pi(1)}(q^{\pi}_k) \leq \cdots \leq \psi_{\pi(n)}(q^{\pi}_k)
\quad \text{for all } k\in\N.
\end{equation}
Set $q^\pi_0:=0$, and for each $\pi\in\Pi$ define
$$
\psi^{\pi}_i(q):=\psi_i(q^{\pi}_{k})  \quad \text{for } q^{\pi}_{k-1} < q \leq q^{\pi}_{k}.
$$
Then for each $\pi\in\Pi$ each function $\psi^\pi_i$ is non-increasing,
\begin{equation}\label{eq: psi chainpi}
    \psi^\pi_{\pi(1)}(q)\leq \cdots\leq \psi^\pi_{\pi(n)}(q)\qquad\text{for all }q,
\end{equation}
and furthermore 
\begin{equation}\label{eq2.6}
\psi^\pi_i(q)\leq \psi_i(q)\quad\text{for all $q\geq 1$}\qquad\text{and}\qquad \psi^\pi_i(q^\pi_k)= \psi_i(q^\pi_k)\quad\text{for for all $k$.} 
\end{equation}
Hence, since $\Pi$ is obviously finite, by \eqref{eq2.3-} and \eqref{eq2.6}, for each $\pi\in\Pi$,
\begin{equation}\label{eq2.7}
\mathcal{S}_{n}(\psi^\pi_1,\ldots,\psi^\pi_n)
\subset
\mathcal{S}_{n}(\psi_1,\ldots,\psi_n).
\end{equation}
Moreover, by \eqref{eq2.3-},
\begin{equation}\label{eq2.8}
\mathcal{S}_{n}(\psi_1,\ldots,\psi_n)
=
\bigcup_{\pi\in\Pi}\mathcal{S}_{n}(\psi^\pi_1,\ldots,\psi^\pi_n).
\end{equation}

Now, if the divergence sum condition in Theorem~\ref{thm: dream} is satisfied, then, by \eqref{eq2.3-} and \eqref{eq2.6}, there exists some $\pi\in\Pi$ such that
$$
\sum_{q=1}^\infty \psi_1^\pi(q)\cdots \psi_n^\pi(q)\ge \sum_{k=1}^\infty \psi_1(q_k^\pi)\cdots \psi_n(q_k^\pi)=\infty,
$$
and, by \eqref{eq2.7}, we may work with $\mathcal{S}_{n}(\psi^\pi_1,\ldots,\psi^\pi_n)$ instead of $\mathcal{S}_n(\psi_1,\dots,\psi_n)$, or equivalently, in view of \eqref{eq: psi chainpi}, we may assume \eqref{eq: psi chain} after permuting the approximation functions. Note that such a permutation preserves the nondegeneracy of $F$.

In turn, if the convergence sum condition in Theorem~\ref{thm: dream} is satisfied, then, by \eqref{eq2.6}, for all $\pi\in\Pi$
$$
\sum_{q=1}^\infty \psi_1^\pi(q)\cdots \psi_n^\pi(q)\le \sum_{q=1}^\infty \psi_1(q)\cdots \psi_n(q)<\infty,
$$
and, by \eqref{eq2.8}, we may work with each individual set $\mathcal{S}_{n}(\psi^\pi_1,\ldots,\psi^\pi_n)$ for fixed $\pi$ instead of $\mathcal{S}_n(\psi_1,\dots,\psi_n)$, or again equivalently we may assume \eqref{eq: psi chain} after permuting the approximation functions.

\medskip

Next, while proving the convergence case of Theorem~\ref{thm: dream}, we may additionally assume, without loss of generality, that for any fixed $\mathfrak c>0$,
\begin{equation}\label{eq: prod is greater than}
\psi_1(q)\cdots\psi_n(q)\, q > q^{-\mathfrak c}
\qquad\text{for all } q\ge1.
\end{equation}
To justify this assumption we will use the following 

\begin{proposition}\label{prop2.1}
Let $\psi_1,\dots,\psi_n:\N\to(0,1)$ be non-increasing and satisfy the chain condition
\eqref{eq: psi chain}. Let $\Phi:\N\to(0,1)$ be any non-increasing function. Then, there exist non-increasing functions
$\psi'_1,\dots,\psi'_n:\N\to(0,1)$ such that for all $q\ge 1$
\begin{equation}\label{eq2.6++}
\psi_i(q)\le \psi_i'(q)\qquad(1\le i\le n),
\end{equation}
\begin{equation}\label{eq2.6+++}
\prod_{i=1}^n \psi_i'(q)
=\max\left\{\prod_{i=1}^n\psi_i(q),\,\Phi(q)\right\},
\end{equation}
and
\begin{equation}\label{eq: psi chain'}
    \psi'_1(q)\leq \cdots\leq \psi'_n(q).
\end{equation}
\end{proposition}

\medskip

With this proposition at hand, justifying \eqref{eq: prod is greater than} becomes simple. We take $\Phi(q)=q^{-1-\mathfrak{c}}$ and define $\psi'_i$ from the proposition. Then
\begin{equation}\label{eq2.13}
\sum_{q=1}^\infty \psi'_{1}(q)\cdots\psi'_{n}(q)
\leq \sum_{q=1}^\infty \psi_{1}(q)\cdots\psi_{n}(q)
+\sum_{q=1}^\infty q^{-1-\mathfrak{c}}
<\infty
\end{equation}
and
$$
\mathcal{S}_n(\psi_1,\dots,\psi_n)
\subset \mathcal{S}_n(\psi_1',\dots,\psi_n').
$$
Hence we may work with the set
$\mathcal{S}_n(\psi_1',\dots,\psi_n')$
instead of $\mathcal{S}_n(\psi_1,\dots,\psi_n)$,
or equivalently we may assume \eqref{eq: prod is greater than}.

\begin{proof}[Proof of Proposition~\ref{prop2.1}]
For $n=1$ the proof is trivial as $\psi'_1$ can be simply defined by equation \eqref{eq: psi chain'}. Clearly, this defines a  nonincreasing function satisfying \eqref{eq2.6++}, with \eqref{eq: psi chain'} becoming trivial. For $n>1$ the existence of the functions $\psi_i'$ is non-trivial. We will establish it on modifying the functions
$\psi_i$ inductively. We begin by setting
$$
\psi_i'(1):=\psi_i(1)\qquad(1\le i\le n).
$$
Let $q\ge1$ and suppose that the values of non-increasing functions $\psi_i'$ have been defined for all
arguments up to $q$ in such a way that \eqref{eq2.6++}, \eqref{eq2.6+++} and
\eqref{eq: psi chain'} hold. Let
\begin{equation}\label{eq2.13+xy}
\delta_0:=\prod_{i=1}^n\psi'_i(q),\qquad
\delta_1:=\prod_{i=1}^n\psi_i(q+1),\qquad
\delta_2:=\Phi(q+1).
\end{equation}
To define the functions $\psi'_i$ at $q+1$ we distinguish three cases.

\medskip

\noindent\textbf{Case 1:}
$\delta_0\le\max\{\delta_1,\delta_2\}$.
In this case we set
\begin{equation}\label{eq2.12vb}
\psi_i'(q+1):=\psi_i'(q)
\qquad(1\le i\le n).
\end{equation}
Since the functions $\psi_i$ are non-increasing, \eqref{eq2.6++} at $q+1$ follows from \eqref{eq2.12vb} and \eqref{eq2.6++} at $q$, and \eqref{eq: psi chain'} at $q+1$ follows from \eqref{eq2.12vb} and \eqref{eq: psi chain'} at $q$. By \eqref{eq2.12vb} and the assumption of this case, we have an upper bound in \eqref{eq2.6+++} at $q+1$. Finally note that the right hand side of \eqref{eq2.6+++} is non-increasing as the maximum of two non-increasing functions. Hence, the lower bound in \eqref{eq2.6+++} at $q+1$ follows from \eqref{eq2.12vb} and \eqref{eq2.6+++} at $q$. Thus,  \eqref{eq2.6+++} holds at $q+1$.

\medskip

\noindent\textbf{Case 2:}
$\delta_0>\delta_1\ge\delta_2$. Then we define
$$
\psi'_i(q+1):=\psi_i(q+1).
$$
This definition satisfies all the required conditions. Indeed,
\eqref{eq2.6++} at $q+1$ holds, and the monotonicity condition follows from
$$
\psi'_i(q+1)=\psi_i(q+1)\le\psi_i(q)\le\psi'_i(q).
$$
Condition \eqref{eq: psi chain'} at $q+1$ follows from \eqref{eq: psi chain};
and finally \eqref{eq2.6+++} at $q+1$ holds since
$$
\prod_{i=1}^n\psi'_i(q+1)=\prod_{i=1}^n\psi_i(q+1)=\delta_1=\max\{\delta_1,\delta_2\}.
$$

\medskip

\noindent\textbf{Case 3:}
$\delta_0>\delta_2>\delta_1$. Then we consider the rectangle
$$
R:=\prod_{i=1}^n\big[\psi_i(q+1),\,\psi_i'(q)\big].
$$
Since $\psi_i(q+1)\le\psi_i(q)\le\psi'_i(q)$, the rectangle $R$ is non-empty, with
some sides possibly reduced to a single point. Clearly, $R$ is a convex subset
of $\R^n$. Define also
$$
C:=\{(\xi_1,\dots,\xi_n)\in\R^n:\xi_1\le\xi_2\le\cdots\le\xi_n\},
$$
which is again convex. Hence $R\cap C$ is convex.
Furthermore,
$$
(\psi_i(q+1))_{i=1}^n\in R\cap C
\qquad\text{and}\qquad
(\psi'_i(q))_{i=1}^n\in R\cap C.
$$
By convexity, the line segment $L$ connecting these points is contained in
$R\cap C$. As $(\xi_i)_{i=1}^n$ varies along $L$, the continuous function
$$
(\xi_1,\dots,\xi_n)\longmapsto \prod_{i=1}^n \xi_i
$$
takes all values in the interval
$
I:=[\delta_1,\delta_0].
$
Since $\delta_2\in I$, there is a point $(\xi_i)_{i=1}^n\in L\subset R\cap C$
such that $\prod_{i=1}^n\xi_i=\delta_2$. We then define
$$
\psi'_i(q+1):=\xi_i\qquad(1\le i\le n).
$$
This choice satisfies all the required conditions.
Indeed, \eqref{eq2.6++} at $q+1$ holds since $(\xi_i)_i\in R$; the monotonicity
condition for $\psi'_i$ follows from the inequality $\xi_i\le\psi'_i(q)$; condition \eqref{eq: psi chain'} at $q+1$
holds since $(\xi_i)_i\in C$, and \eqref{eq2.6+++} holds since
$$
\prod_{i=1}^n\psi'_i(q+1)=\delta_2=\max\{\delta_1,\delta_2\}.
$$
This completes the induction step. Finally, since $\psi'_i$ are non-increasing and
$\psi'_i(1)=\psi_i(1)$ are between $0$ and $1$, we have that $\psi'_i:\N\to(0,1)$.
\end{proof}

\subsection{A special parametrization}\label{sec: parametrization} 
We will start by  giving a proof of Theorem~\ref{thm: dream} for maps
$F:\bU \subset \R^d \to \R^n$ of the form
\begin{equation}\label{eq: F}
F(\bx) = (\bx, f(\bx)), \qquad \bx \in \bU, \quad
f(\bx) = (f_i(\bx))_{i=1}^m,
\end{equation}
where $m = n - d$. In general, a nondegenerate map $F$ cannot always be reduced
to this form by a change of variables while preserving the chain condition
\eqref{eq: psi chain}. Consequently, in \S\ref{sec: general} we extend the proof
from this special case to the general setting, thereby establishing the result
for all nondegenerate maps $F$.

Also, we assume without loss of generality that there exists $M>0$ such that 
\begin{equation}\label{eq: bound on derivatives}
   \max_{k}\max_{i,j} \sup_{\bx\in \bU}\vert\partial_{ij}f_{k}(\bx)\vert\leq M.
\end{equation}

\subsection{Quantitative Nondivergence}\label{new BKM} 

Our proofs will require an effective measure estimate for the following set 
\begin{align}
\mathfrak{S}_{F}&(\delta,\bK,\bT)=\nonumber\\[1ex]
&=\left\{\bx\in \bU:\exists\;(a_0,\ba)\in\Z\times\Z^n_{\neq\bf0}\;\;\text{such that }\left.
\begin{array}{l}
|a_0+F(\bx)\ba^\top|<\delta\\[1ex]
\vert \partial_iF(\bx)\ba^\top\vert<K_i\;\;(1\le i\le d) \\[1ex]
\vert a_k\vert<T_k ~~(1\leq k\leq n) 
\end{array}
\right.\right\},\label{eq: new set}
\end{align}
where $\ba=(a_1,\dots,a_n)$, $\bT=(T_1,\dots,T_n)$, $\bK=(K_1,\dots,K_d)$, and $\delta,K_1,\cdots,K_d,T_1,\cdots, T_n$ are positive real parameters.

The following result represents a generalisation of \cite[Theorem 1.4]{BKM}
and appears in \cite[Theorem 3.1]{BD} when $K_1=\dots=K_d$. Its proof in the general case follows the same lines as that of \cite[Theorem 3.1]{BD} with minor modifications which will be explained below.

\begin{theorem}\label{BKM_new2}
Let $F:\bU\to\R^n$ be as defined in \eqref{eq: F}, $l$-nondegenerate at $\bx_0$. 
Then there exists a ball $B_0\subset\bU$ centred at $\bx_0$ and a constant $E>0$ depending on $B_0$ and $F$ such that for any choice of $\delta,K_1,\cdots,K_d,T_1,\cdots, T_n$ satisfying 
\begin{equation}\label{eqn5.2}
0<\delta\le 1,\qquad T_1,\dots,T_n\ge1,\qquad K_1,\cdots,K_d>0 \quad\text{and}\quad \delta^{n}<K \frac{T_1\cdots T_n}{\max_i T_i},
\end{equation}
where $K:=\max\{K_1,\dots,K_d\}$, and any ball $B\subset B_0$ of radius $0<r\le1$ we have that
\begin{equation}\label{eqn5.3_0}
\begin{array}[b]{l}
\displaystyle\rule{0ex}{3ex}\mu_d\big(\mathfrak{S}_{F}(\delta,\bK,\bT)\cap B\big)\;\\[1ex] \displaystyle\ll \delta \prod_{i=1}^d\min\{K_i,T_i\}\cdot(T_{d+1}\cdots T_n)\mu_d(B)+ E\left(\delta \min\{K,r^{-1}\}\frac{T_1\cdots T_n}{\max_i T_i}\right)^{\alpha},
\end{array}
\end{equation}
where $\alpha=\frac{1}{d(2l-1)(n+1)}$ and the implied constant depends on $l$, $m$ and $d$ only.
\end{theorem}

\begin{proof}
Following the proof of \cite[Theorem 3.1]{BD}, when $K \geq r^{-1}$, we have that
$$
\begin{aligned}
& \mathfrak{S}_{F}(\delta,\bK,\bT)\setminus \mathfrak{S}_{F}(\delta,r^{-1},\cdots,r^{-1},\bT)\\
& = \bigcup_{\ba} \left\{\bx\in \bU~|~\exists~ a_0\in \Z \text{ such that } \begin{aligned}
    &\vert a_0+ F(\bx)\ba^\top\vert<\delta\\
    & \Vert \nabla F(\bx)\ba^\top\Vert> r^{-1}
\end{aligned}
\right\}
\end{aligned},$$ where $\ba\in \Z^n$ be such that for some $\bx\in B_0$,
$$\vert \partial_{i}F(\bx)\ba^\top\vert< K_i\qquad (1\leq i\leq d).$$

Since $F$ is parametrized as  \eqref{eq: F}, we get that for fixed $a_{d+1},\cdots, a_{n}$ there are $\prod_{i=1}^d\max\{K_i, T_i\}$ many choices for $a_1,\cdots, a_d.$ This, together with \cite[Theorem 4]{BD} gives the first term in \eqref{eqn5.3_0}, and the second term follows from \cite[Theorem 1.4]{BKM}.
\end{proof}

\subsection{Geometry of numbers} Given a lattice $\Lambda$ in $\R^{n+1}$, 
$$
\lambda_1(\Lambda),\dots,\lambda_{n+1}(\Lambda)
$$
will denote Minkowski's {\em successive minima} of the  the lattice $\Lambda$ with respect to the closed unit Euclidean ball centered at $\mathbf{0}$. For any $g\in \GL_{n+1}(\R),$ we define the dual matrix $g^\star=\sigma^{-1}(g^{\top})^{-1}\sigma$, where $\sigma$ is the long Weyl element. We refer the reader to \cite[\S 5.1]{BD} for more details. 
For $g\in \GL_{n+1}(\R),$ by Minkowski's second theorem, we have that 
$$\lambda_{1}(g\Z^{n+1}) \cdots \lambda_{n+1}(g\Z^{n+1})\asymp_n \vert\det g\vert\,.
$$
Now we recall \cite[Lemma 3.3]{BY}.

\begin{lemma}\label{dual lemma}
    Let $g\in \mathrm{GL}_{n+1}(\R)$ then 
    \begin{equation}
        \lambda_1(g\Z^{n+1}) \lambda_{n+1}(g^\star \Z^{n+1}) \asymp_{n} 1. 
    \end{equation}
\end{lemma}

\subsection{Ubiquitous systems for rectangles}\label{sec: abstract}
In this subsection we recall the theory developed by Kleinbock and Wang in \cite{KWang23}.
For $1\leq i\leq d$, let $U_i$ be an open subset of $\R$, $U=U_1\times\cdots\times U_d$, and let $\mathcal{R}_{\alpha}=(\mathcal{R}_{\alpha,i})_{i=1}^d$ be a family of singletons such that $\mathcal{R}_{\alpha,i}\in U_i$ for each $i$ and $\alpha$.
Here the index $\alpha$ runs over an infinite countable set $J.$ Further, let $J\ni \alpha\mapsto \beta_{\alpha}$ be a positive valued map such that for any $M>1$ the set $\{\alpha\in J: \beta_{\alpha}\leq M\}$ is finite. Let us fix an increasing sequence $(u_t)_{t\ge0}$ of positive integers, and for each $t\in\N$ define
$$J_t:=\{\alpha\in J~|~u_{t-1}\leq \beta_\alpha\leq u_t\}.
$$
For any $\vv r=(r_i)_{i=1}^d\in \Rp^d$ and any $\mathcal{R}_\alpha$, we define the rectangle
\begin{equation}\label{Delta1}
\tDelta(\mathcal{R}_{\alpha}, \mathbf{r}):=\prod_{i=1}^d B(\mathcal{R}_{\alpha, i}, r_i).
\end{equation}

\begin{definition}
Let $\brho:=(\rho_i)_{i=1}^d:\R_{>0}^d\to \R_{>0}$ be such that $\lim_{t\to\infty} \brho(t)=\mathbf{0}$. We call $\{\mathcal{R}_{\alpha}; \beta\}$  ubiquitous in $U$ with respect to $\brho $ and $(u_t)_t$ if there exists a constant $k_0>0$ such that for any ball $B\subset U$    \begin{equation}\label{eq: measure large of ubi}
        \mu_d\left(B\cap \bigcup_{\alpha\in J_t}\tDelta(\mathcal{R}_{\alpha}, \brho(u_t))\right)\geq k_0\mu_d(B)\qquad\text{for all } t\geq t_0(B).
    \end{equation}
\end{definition}

Let 
\begin{equation}\label{bPsi}
\bPsi := (\Psi_i)_{i=1}^d : \R_{>0} \to \R_{>0}^d
\end{equation}
and
$$\Lambda_{\mathcal{R}}(\bPsi)=\{\bx\in U~|~\bx\in \tDelta(\mathcal{R}_{\alpha} \bPsi(\beta_{\alpha})) \text{ for infinitely many }\alpha\in J\}. $$
Now we recall Theorem 2.5 form \cite{KWang23}.

\begin{thkw}\label{thm:abstract}
Suppose that $\{\mathcal{R}_{\alpha}; \beta\}$ is ubiquitous in $U$ with respect to $\brho$ and $\{u_t\}$. Suppose that a map \eqref{bPsi}
is given, each $\Psi_i$ is monotonic and $\Psi_i(u_t) \leq \rho_i(u_t)$ for all sufficiently large $t$. Further suppose that there exists $0 < \lambda < 1$ such that
$$
\Psi_i(u_{t+1}) \leq \lambda \Psi_i(u_t)
$$
for all sufficiently large $t$ and all $1 \leq i \leq d$. Then
\begin{equation}\label{eq:ubiq}
\mu_d\bigl(\Lambda_{\mathcal{R}}(\bPsi)\bigr) = \mu_d(U)
\qquad \text{if} \quad
\sum_{t \geq 1} \prod_{i=1}^d \frac{\Psi_i(u_t)}{\rho_i(u_t)} = \infty.
\end{equation}
\end{thkw}

\subsection{Further notation} 

For every choice of the following parameters
$$
0<\varepsilon_{d+j}<1~~(1\leq j\leq m),\qquad \rho_i>0~~(1\leq i\leq d),\qquad Q>0
$$ 
and a ball $B\subset \bU$, define the following sets
$$
\mathcal{R}(Q, (\varepsilon_{d+j})_{j=1}^m, B):=\left\{(q,\ba,\bb)\in \N\times \Z^{n}~|~\begin{aligned}
    &\left\vert qf_{j}(\ba/q)-b_j\right\vert<\ep_{d+j}, \quad 1\leq j\leq m,\\
    & \ba/q\in B, \quad Q/2\leq q\leq Q
\end{aligned} \right\}$$
and
\begin{equation}\label{def: Delta}
\Delta(Q, (\varepsilon_{d+j})_{j=1}^m, B, \brho):=\bigcup_{(q,\ba,\bb)\in \mathcal{R}(Q, (\varepsilon_{d+j})_{j=1}^m, B)} \tDelta(\ba/q, \brho),
\end{equation} 
where $\brho=(\rho_i)_{i=1}^d$. Also define the following matrices for later use:
\begin{equation}\label{def_u}
u(\bx):=\begin{bmatrix}
    \mathrm{I}_{m} & 0   & -\partial_1 f(\bx) & f(\bx)-  x_1\partial_1 f(\bx)\\
      0 & \mathrm{I}_{d-1} & 0 & (x_d,\dots,x_2)^T \\
      0 & 0 & 1 & x_1\\
    0 & 0  & 0 & 1
\end{bmatrix},\end{equation}

and 

\begin{equation}\label{def_u_1}
u_1(\bx)=\begin{bmatrix}
    1 & 0 & \cdots & 0 & -\partial_d f_m(\bx) & \cdots &  -\partial_1 f_m(\bx) & f_{m}(\bx)- \sum_{i=1}^d x_i\partial_i f_m(\bx)\\
    \vdots & \vdots & \vdots & \vdots & \vdots & \vdots & \vdots & \vdots \\
    0 & 0 & \cdots & 1 & -\partial_d f_1(\bx) & \cdots &  -\partial_1 f_1(\bx) & f_{1}(\bx)- \sum_{i=1}^d x_i\partial_i f_1(\bx)\\
    0 & 0 & \cdots & 0 & 1 & 0 & \cdots & x_{d}\\
     \vdots & \vdots & \vdots & \vdots & \vdots & \vdots & \vdots & \vdots \\
     0 & 0 & \cdots & 0 & 0 &  \cdots & 1 &  x_{1}\\
     0 & 0 & \cdots & 0 & 0 &  \cdots & 0 &  1    
\end{bmatrix}.\end{equation}

\section{Lower bound and divergence: the special case}\label{sec: div}

The purpose of this section is to prove the divergence case of Theorem~\ref{thm: dream} for nondegenerate maps of the form \eqref{eq: F}.
This will be done by applying Theorem~KW stated in \S\ref{sec: abstract}. Hence our key goal of this section is to construct a suitable ubiquitous system.

\subsection{Setting off}

First, we define some relevant parameters and matrices. Let $c>0$, $\bve=(\ve_1,\dots,\ve_n)\in\Rp^n$, $ Q>0$ and define the following diagonal matrix
\begin{equation}\label{def: g}
    g=g(c,Q, \bve,\bve'):=c^{-1}\diag(\varepsilon_n,\dots, \varepsilon_{d+1}, \varepsilon_d', \dots, \varepsilon_1', c^{n+1} Q), 
\end{equation} 
where $\bve'=(\ve'_1,\dots,\ve'_d)\in\Rp^d$ satisfies 
\begin{equation}\label{eq: ep'}
    \varepsilon'_1\cdots \varepsilon'_d= (\varepsilon_{d+1}\cdots \varepsilon_{n}Q)^{-1}.
\end{equation}
For the rest of the section we will assume that 
\begin{equation}\label{eq: chain}
    0<\varepsilon_1\leq \varepsilon_{2}\leq \cdots\leq \varepsilon_{n}\le 1.
\end{equation}
Define the set 
\begin{equation}\label{def: mathcal{G}}
\mathcal{G}=\mathcal{G}(c,Q, \bve,\bve'):=\{\bx\in \bU: \lambda_{n+1}(g^{-1} u(\bx)\Z^{n+1})\leq c^{-n}\},
\end{equation} 
where $u(\bx)$ is defined in \eqref{def_u}. Note that in $u(\bx)$, only one directional derivative is used. 

The choice of $\bve'$ which will determine the ubiquity function $\brho$ satisfying \eqref{eq: measure large of ubi} on the set of rational points lying appropriately close to a given manifold. With this in mind, we now gather together three conditions that we will need to meet:
\begin{equation}\label{Goal1}
\ep_{i}'\geq \ep_i\qquad (1\leq i\leq d);
\end{equation}
\begin{equation}\label{Goal2}
       \varepsilon_{d+1}>\frac{1}{Q}\max_{1\leq i, i'\leq d}\varepsilon_i' \varepsilon'_{i'};
\end{equation}
\begin{equation}
\label{Goal3}
    Q^{-n}< \frac{\max_{i=1}^d K_i}{\max_{i=1}^n T_i} T_1\cdots T_n,
\end{equation} 
where parameters $K_i$ and $T_i$ are given by \eqref{def: KT} and dependent on $\bve'$.
We will explain in Remark~\ref{goal} below the relevance of these conditions to our proof.

\subsection{Measure bound}\label{sec: modi_BKM}
The goal of this section is to show that the set $\mathcal{G}$, where $\mathcal{G}$ is given by \eqref{def: mathcal{G}}, is relatively large (equivalently its complement is relatively small). To begin with, using Lemma \ref{dual lemma}, we get that 
$$\begin{aligned}
    \bU\setminus \mathcal{G}\subset\{\bx\in \bU~|~\lambda_1((g^{-1})^\star u^\star(\bx)\Z^{n+1})\ll_{n} c^n\},
\end{aligned}
$$
Note that 
\begin{equation}\label{eqn:dual u_1(x)}
        u^\star(\bx)=\begin{bmatrix}
            1 & -\bx & -f(\bx)\\
            0 & \be_1 & \partial_1f(\bx)\\
            0 & 0   & \mathrm{I}_{n-1}
        \end{bmatrix}\,,
    \end{equation} 
where $\be_1=(1,0,\dots,0)\in \R^d.$ Also, note that 
$$(g^{-1})^{\star}=\sigma g\sigma=c^{-1}\diag(c^{n+1}Q,\varepsilon_1', \cdots, \varepsilon_d', \varepsilon_{d+1},\cdots, \varepsilon_{n}).
$$
Using \eqref{eq: chain}, we get that $$\{\bx\in \bU~|~\lambda_1((g^{-1})^\star u^\star(\bx)\Z^{n+1})\ll_{n} c^n\}
\subset \mathfrak{S}_{F}(Q^{-1}, K_1,\cdots,K_d, T_1,\cdots,T_n),
$$
where $\mathfrak{S}_{F}(\cdots)$ is defined in \S\ref{new BKM}  and  \begin{equation}\label{def: KT}
\begin{aligned}
&K_1\asymp_{n}c^{n+1}\varepsilon_1'^{-1},\\ 
& T_i \asymp_{n}c^{n+1}\varepsilon_i'^{-1}\hspace{15ex} (2\leq i\leq d), \\
& T_{d+j} \asymp_{n}c^{n+1}\varepsilon_{d+j}^{-1}\hspace*{12ex} (1\leq j\leq m),\\
& T_1 \asymp_{n, F}c^{n+1}\varepsilon_1'^{-1}+ c^{n+1}m\varepsilon_{d+1}^{-1}, \\
& K_i\asymp_{n} T_i+ c^{n+1}\ep_{d+1}^{-1}\hspace*{9ex} (2\leq i\leq d).
\end{aligned}
\end{equation}
Observe that 
\begin{equation}\label{eq3.9}
\max_{i=1}^n T_i\asymp_{n, F}\max_{i=1}^d\{\ep_i'^{-1}, \ep_{d+1}^{-1}\}.
\end{equation}
Thus, assuming \eqref{Goal3} holds, on applying Theorem~\ref{BKM_new2} we obtain that the $\mu_d$ measure of 
$$
\mathfrak{S}_{F}(Q^{-1}, K_1,\cdots,K_d, T_1,\cdots,T_n)$$ intersected with any ball $B\subset B_0$ (where $B_0$ is as in Theorem~\ref{BKM_new2}) is bounded above by
\begin{equation}\label{eq: measure bound}
\begin{aligned}
    &\underbrace{(Q^{-1}) (c^{d(n+1)} \prod_{i=1}^d\varepsilon_i'^{-1} ) c^{m(n+1)}\prod_{j=1}^m\varepsilon_{d+j}^{-1} \mu_d(B)}_{M_1}\\
    &+\underbrace{E \left(c^{(n+1)n}Q^{-1}r^{-1}(B)  \frac{(\varepsilon_1'^{-1}+ \varepsilon_{d+1}^{-1})}{\max_{i=1}^d\{\varepsilon_i'^{-1}+ \varepsilon_{d+1}^{-1}\}} \prod_{i=2}^d\varepsilon_i'^{-1}\prod_{j=1}^m\varepsilon_{d+j}^{-1}\right)^{\alpha}}_{M_2},
\end{aligned}
\end{equation}
upto some constant that depends on $l,m,d, F$, 
where $E>0, \alpha$ are the constants as in Theorem~\ref{BKM_new2}. Here we used that $\varepsilon_{d+1}^{-1}\geq\ep_{d+j}^{-1}$ implied by \eqref{eq: chain}, so that $\max_{i=1}^n T_i=\max_{i=1}^d T_i$. For the rest of this section $E, \alpha$ are as in Theorem~\ref{BKM_new2}.

Using \eqref{eq: ep'}, we obtain that
\begin{equation}\label{eq: M_1}
    M_1\leq c^{n(n+1)}\mu_d(B).
\end{equation}
From now on, given $\bve=(\ep_i)_{i=1}^n$ satisfying \eqref{eq: chain}, we choose \begin{equation}\label{eq: simple ep'}
    \begin{aligned}
    \ep_1'= \frac{\ep_1}{\ep_1\cdots\ep_n Q},\qquad\text{and}\qquad\ep_i'=\ep_i\quad (2\leq i\leq d).
    \end{aligned}
\end{equation}

\begin{remark}\label{rem:goal 1}
Note that \eqref{Goal1} is satisfied with the above choice of $\ep_i'$ if and only if
\begin{equation}\label{eq: standing}
    \ep_1\cdots\ep_nQ\le 1.
\end{equation}
\end{remark}

We now arrive at the key proposition of this subsection:

\begin{proposition}\label{prop: crucial in Case 2} Let $F$ be $l$-nondegenerate at $\bx_0$ as defined in \eqref{eq: F}. Then there exists a ball $B_0\subset \bU$ centered at $\bx_0$ and $0<c<1$ such that for any ball $B\subset B_0$, $0<s<\frac{1}{n^2}$, there exists $Q_0=Q_0(B,F, n, d, s)$ such that for all $Q> Q_0,$ for any $\bve=(\ve_i)_{i=1}^n$ satisfying \eqref{eq: chain} and \begin{equation}\label{eq: standing 2}
     \ep_2\cdots\ep_{n}Q>Q^{s},
     \end{equation}
we have that $$\mu_d((\bU\setminus \mathcal{G})\cap B)\leq  \frac{1}{2^{2d}}\mu_d(B),
$$ 
where $\mathcal{G}=\mathcal{G}(c,Q, \bve,\bve')$ is given by \eqref{def: mathcal{G}} with $\bve'=(\ep_i')_{i=1}^d$ as in \eqref{eq: simple ep'}.
\end{proposition}

\begin{proof}
We continue with the notation and discussion from \S \ref{sec: modi_BKM}. By the choice of $\ep_i'$ in \eqref{eq: simple ep'},
and by \eqref{eq: standing 2}, we have that $\ep_1'\leq 1$.
Also, by \eqref{eq: chain}, we get 

$$\max_{i=1}^d K_i=c^{n+1} \max \{\ep_1'^{-1}, \ep_2^{-1}\}\asymp_{n}
\max_{i=1}^n T_i.$$
The above and the fact that $T_i\geq 1$ for each $1\leq i\leq n$ verify \eqref{Goal3}. Hence we can apply Theorem~\ref{BKM_new2}, and get 
$$
M_2~~{\ll}_{F}~~ E \left( c^{(n+1)^2} r(B)^{-1} Q^{-1} \prod_{i=2}^n\ep_{i}^{-1}\right)^{\alpha},
$$ 
where $\alpha=\frac{1}{d(2l-1)(n+1)}$.
Then, using \eqref{eq: standing 2}, we obtain 
$$M_2 \ll_{n, F, B_0} (r(B)^{-1} Q^{-s})^{\alpha} c^{(n+1)n\alpha}.$$   Takings  $Q>Q_0(B, F, n, d, s)$ we get 
$$
M_2 ~\leq~ \frac{1}{2^{2d+2}}\mu_d(B).
$$
Also, by  \eqref{eq: M_1}, we choose $c<1$ small enough (only depending on $n$) such that 
$$
M_1 ~\leq~ \frac{1}{2^{2d+2}}\mu_{d}(B).$$ Combining the estimates for $M_1$ and $M_2$  completes the proof.
\end{proof}

\subsection{Projection lemma}\label{sec: projection}

The goal now is to establish the ubiquity of rational point near a given manifold projected to $\R^d$. The following statement is proved exactly the same way as in \cite[Lemma 2.1]{BVVZ21}. In what follows $V$ is an open set in $\bU\subset\R^d$, where $\bU$ be as in \S \ref{sec: parametrization}. Given $\br=(r_i)_{i=1}^d\in\Rp^d$, we define $V^{(\br)}$ as the set of points $\bx$ such that $\prod_{i=1}^d B(x_i,r_i)\subset V$. 

\begin{lemma}\label{lem: projection}
Let $V\subset \bU$. Let $Q>0, c<1,$ $\bve=(\varepsilon_i)_{i=1}^n$ satisfy \eqref{eq: chain}, $\bve'=(\ep_i')_{i=1}^d\in\Rp^d$ be as in \eqref{eq: ep'} satisfying \eqref{Goal2}, that is
\begin{equation*}
       \varepsilon_{d+1}\stackrel{\eqref{eq: chain}}{=} \min_{j=1}^m \varepsilon_{d+j}>\frac{1}{Q}\max_{1\leq i, i'\leq d}\varepsilon_i' \varepsilon'_{i'},
    \end{equation*} and
\begin{equation}\label{eq:epiep_{d+1}}
        \ep_{i}'\leq \ep_{d+1}\qquad (2\leq i\leq d).
\end{equation}
Let $r_i=\frac{\varepsilon'_i}{c^{n+1}Q}$ $(1\leq i\leq d)$. Then for any $\bx\in \mathcal{G}\cap V^{(\br)}$ there exists integer points 
    $$
    (q, \underbrace{a'_1,\cdots, a'_d}_{\ba'},\underbrace{b_1,\cdots, b_m}_{\bb})\in \Z^{n+1}
    $$ 
    such that 
    \begin{equation}\label{eq: linear_div}
    \begin{aligned}
    &\vert qx_i-a'_i\vert\ll_{n} \frac{n+1}{c^{n+1}} \varepsilon_i' \quad 1\leq i\leq d,\\
& \vert qf_{j}(\ba'/q)-b_j\vert\leq \frac{(n+1)M}{c^n}\varepsilon_{d+j} \quad (1\leq j\leq m),\\
& (n+1)Q\leq q\leq 3(n+1)Q,
    \end{aligned}     
    \end{equation}
    where $M$ is as in \eqref{eq: bound on derivatives}.
\end{lemma}
\begin{proof}
By definition, $\lambda_{n+1}(g^{-1} u(\bx)\Z^{n+1})\leq c^{-n}$. This means there exist linearly independent vectors  $\{\bv_i\}_{i=1}^{n+1}$ of the lattice $g^{-1}u(\bx)\Z^{n+1}$ lying in the ball of radius $c^{-n}$. For any $\bw=(w_i)\in\R^{n+1}$, we can write 
$$
g^{-1}u(\bx)\bw=\sum_{i=1}^{n+1}\eta_i\bv_i.
$$ 
In particular, we choose 
$$
\begin{aligned}
&w_{n+1}=-2(n+1)Q,\\ 
&w_{m+i}=-w_{n+1}x_{d+1-i}&&(1\leq i\leq d)\\
&w_{j}=-w_{n+1}f_{m+1-j}(\bx)&&( 1\leq j\leq m).
\end{aligned}
$$
Also, set $t_i=\lfloor \eta_i\rfloor$, and let  $(\ba,q)\in\Z^{n+1}$ be such that $\sum t_i\bv_i=-g^{-1}u(\bx)(\ba,q).$ Thus we get 
$$\Vert g^{-1}u(\bx)((\ba,q)+\bw)\Vert\leq (n+1) c^{-n}.$$ 
From the last coordinate of the vector in the left hand side above, we get the last inequality in \eqref{eq: linear_div}. The bound on coordinates $(m+1)$ to $n$ of the above vector give 
$$
\vert qx_i+a_{n+1-i}\vert =\vert (q+w_{n+1})x_i+(a_{n+1-i}-w_{n+1}x_i)\vert <c^{-n-1}(n+1)\ep_i'\quad(1\leq i\leq d).
$$
This gives the first inequality in \eqref{eq: linear_div} where $a_i'=-a_{n+1-i}$ for $1\leq i\leq d$.
The bound on the first $m$ coordinates of the vector gives 
$$\begin{aligned}&\vert (q+w_{n+1})f_{j}(\bx)-\partial_{1}f_{j}(\bx)((q+w_{n+1})x_i+(a_{n+1-i}+w_{n+1-i}))+(a_{m+1-j}+w_{m+1-j})\vert\\
&\leq \frac{n+1}{c^{n+1}}\ep_{d+j}.\end{aligned}$$

After plugging the values of the coordinates of $\bw$ and by \eqref{eq:epiep_{d+1}}, \eqref{eq: chain} we get, 
$$\vert qf_{j}(\bx)-\sum_{i=1}^d \partial_{i}f_{j}(\bx)(qx_i+a_{n+1-i})+a_{m+1-j}\vert\leq \frac{n+1}{c^{n+1}}\ep_{d+j}.$$

Note from the first and the last inequality in \eqref{eq: linear_div}, we get that, since $\bx\in V^{(\br)}$, 
$$
(-a_{n+1-i}/q)_{i=1}^d\in V\subset \bU.
$$

Now using Taylor's expansion, $$\begin{aligned}&\vert q( f_j(-(a_{n+1-i}/q)_{i=1}^d)+a_{m+1-j}/q)\vert\\
& \leq \frac{n+1}{c^{n+1}}\ep_{d+j}+ M \vert q\vert \max_{1\leq i,i'\leq d}\vert x_i+a_{n+1-i}/q\vert \vert x_{i'}+a_{n+1-i'}/q\vert \\
& \leq \frac{n+1}{c^{n+1}}\ep_{d+j}+ \frac{M}{(n+1)Q}\frac{(n+1)^2}{c^{2(n+1)}}\max_{1\leq i,i'\leq d}{\ep_{i}'}{\ep_{i'}'}\\
&\leq \frac{M(n+1)}{c^{n+1}}\ep_{d+j}.
\end{aligned}$$
The last inequality follows from \eqref{Goal2}.
\end{proof}

We now arrive at the key statement of this subsection. We will use $\Delta(\cdot)$ defined in \eqref{def: Delta}.

\begin{proposition}\label{thm: ubiq main}
Let $F$ be given by \eqref{eq: F} and $l$-nondegenerate at $\bx_0$. Then, there exists a ball $B_0\subset \bU$ centered at $\bx_0$ and $\rho_0>0$ with the following property. For any open ball $B\subset B_0$, $0<s<\frac{1}{n^2}$ there exists constants $Q_0(B, F, n, d, s)$ such that for all $Q\geq Q_0$, all 
$\bve=(\ep_{i})_{i=1}^n$ satisfying \eqref{eq: chain} and \eqref{eq: standing 2},
we have that
\begin{equation}
    \mu_d(\Delta(Q, (\ep_{d+j})_{j=1}^m, B, \brho)\cap B)\geq \frac{1}{2}\mu_d(B),
\end{equation} where $\rho_i=\rho_0 \ep_{i}'/Q$, where $\ep_i'$ are defined in \eqref{eq: simple ep'}.
\end{proposition}

\begin{proof}
   Suppose we choose $0<c<1$ to be smaller than as in Proposition \ref{prop: crucial in Case 2}. By \eqref{eq: standing 2} we have that 

   $$\ep_2 Q\geq \ep_2\cdots\ep_{n} Q>Q^{s}\implies \ep_2\geq \frac{Q^{s}}{Q}.$$ Now by \eqref{eq: chain}, \eqref{eq: simple ep'} and \eqref{eq: standing 2},
$$\frac{1}{Q}\ep_1'< \frac{1}{Q^{s+1}}<\ep_2.$$
The above together with $\ep_1'\leq 1, \ep_i\leq 1$ and \eqref{eq: chain} verifies \eqref{Goal2}. Thus \eqref{Goal2} is achieved. Also, note by the choice of $\ep_i'$ in \eqref{eq: simple ep'}, \eqref{eq:epiep_{d+1}} holds. Now by Lemma \ref{lem: projection}, we get that
$$
\mathcal{G}\cap B^{(\brho)}\subset \Delta(\tilde{Q}, (\tilde\ep_{d+j})_{j=1}^m, B, \brho),
$$ 
where 
$$
\tilde Q=3(n+1)Q,\quad \tilde\varepsilon_{d+j}= \frac{M(n+1)}{c^{n+1}}\ep_{d+j}\quad (1\leq j\leq m),\quad \rho_i= \rho_0\frac{\ep_i'}{Q}\quad (1\le i\le d)
$$ 
and $\rho_0=\frac{1}{c^{n+1}}$ with $c$ depending on $n, F, d$ only. 
Note that  $$\mu_{d}(B\setminus B^{(\brho)})\ll_{d, B} \max_{i=1}^d\rho_i\stackrel{\eqref{eq: simple ep'}}{=}\rho_0\max_{i=2}^d\frac{\max\{\ep_1',\ep_{i}\}}{Q}.$$ Since by \eqref{eq: standing 2} $\ep_1'=\frac{1}{\ep_2\cdots\ep_n Q}< \frac{1}{Q^s}<1,$ we get from the above that 
$$\mu_{d}(B\setminus B^{(\brho)})\leq \frac{\rho_0}{Q}<\frac{1}{4}\mu_d(B)$$ for large enough $Q\ge Q_2(B, F, n, d).$ Since $\mathcal{G}\cap B\subset B\setminus B^{(\brho)}\cup   \Delta(\tilde{Q}, (\tilde\ep_{d+j})_{j=1}^m, B, \brho) $, the conclusion follows from Proposition \ref{prop: crucial in Case 2}. 
\end{proof}

We now complement Proposition~\ref{thm: ubiq main} with the following counting lower bound.

\begin{corollary}\label{counting lower}
Let $F$ and $\bx_0$ be as in Proposition~\ref{thm: ubiq main}. Then there exists a ball $B_0$ centered at $\bx_0$ such that for any open set $B\subset B_0$,  $0<s<\frac{1}{n^2}$ there exists $Q_0(B, n, F, d, s)$ such that for all $Q\geq Q_0$, $0<\ep_{d+j}<1$ $(1\leq j\leq m)$ with $\ep_{d+1}\leq \cdots\leq \ep_n$, and 
\begin{equation}\label{eq: standing 3}\ep_{d+1}^{d-1}\ep_{d+1}\cdots\ep_n Q> Q^{s},\end{equation} we have that 
$$\# \mathcal{R}(Q, (\varepsilon_{d+j})_{j=1}^m, B)\gg_{F} \ep_{d+1}\cdots\ep_{n} Q^{d+1}\mu_{d}(B).$$

\end{corollary}

\begin{remark}
The monotonicity condition on $(\ep_{d+j})_{j=1}^m$ in the above statement is not restrictive and is used  to write \eqref{eq: standing 3} in a clear form. Without it  condition \eqref{eq: standing 3} would have to be replaced by
   $$(\min_{j=1}^m\ep_{d+j})^{d-1} \ep_{d+1}\cdots\ep_{n}>Q^s.$$

\end{remark}
\begin{proof}
Let $\ep_{d+j}$ $(1\leq j\leq m)$ be given as in the corollary. Define $\ep_1:=\cdots=\ep_{d}:=\ep_{d+1}$. Thus $\bve=(\ep_i)_{i=1}^{n}$ satisfies \eqref{eq: chain}. With this choice of $\bve$, \eqref{eq: standing 2} is the same as \eqref{eq: standing 3}. Proposition \ref{thm: ubiq main} implies 
$$\# \mathcal{R}(Q, (\varepsilon_{d+j})_{j=1}^m, B)\frac{\prod_{i=1}^d \ep_i'}{Q^d}\gg_{F} \mu_{d}(B),$$ where $\ep_i'$ are defined in \eqref{eq: simple ep'}. The conclusion follows since $\prod_{i=1}^d \ep_i'=(\ep_{d+1}\cdots \ep_n Q)^{-1}$. 
\end{proof}

\begin{remark}\label{goal}
We want to point towards the importance and motivation of conditions \eqref{Goal1}--\eqref{Goal3} imposed on $\bve'$. Thanks to \eqref{Goal3} we can use Theorem~\ref{BKM_new2} to get Proposition~\ref{prop: crucial in Case 2}. Using \eqref{Goal2}, we apply Lemma~\ref{lem: projection} (the \textit{projection} lemma) to achieve the crucial Proposition \ref{thm: ubiq main}. For any $\bve=(\ep_i)_{i=1}^n$ satisfying \eqref{eq: chain} and \eqref{eq: standing 2} both of these conditions were established for our chosen $\bve'$. Finally  we established \eqref{Goal1} when $\bve$ also satisfy \eqref{eq: standing}, which is required to apply the ubiquity Theorem KW in \S\ref{sec: abstract}.
\end{remark}

\subsection{Ubiquity setup and limsup sets}\label{sec: finish div}

Let $\psi_1,\dots,\psi_n$ be approximation functions satisfying the divergence sum condition in Theorem~\ref{thm: dream}. Then
\begin{equation}\label{eq3.20}
\sum_{t=1}^\infty 2^t\psi_1(2^t)\cdots\psi_n(2^t)=\infty.
\end{equation}
As explained in \S \ref{sec: con on psi div}, we assume $\psi_i$ satisfy \eqref{eq: psi chain}, without loss of generality.

Fix any $0<s'<\frac{1}{n-1}$. Then, by \eqref{eq3.20}, there exists a strictly increasing sequence of positive integers $(t_k)_{k\in\N}$ such that
\begin{equation}\label{eq:lower bound on prod} 
\psi_{1}(q)\cdots \psi_{n}(q) q>q^{-\left(\frac{1}{n-1}-s'\right)}\qquad\text{for each } q=2^{t_k}
\end{equation}
and 
\begin{equation}\label{eq3.22}
\sum_{k=1}^\infty 2^{t_k}\psi_1(2^{t_k})\cdots \psi_n(2^{t_k})=\infty.
\end{equation}
In particular, we have that for all $k\in\N$
\begin{equation}\label{eq: standing 2 in 2^t}
2^{t_k}\psi_{2}(2^{t_k})\cdots \psi_n(2^{t_k})>2^{st_k},\qquad s:=\frac{n-1}{n}s'.
\end{equation}
Define $\mathcal{T}:=\{t_k:k\in\N\}$, which is thus an infinite collection. Split $\mathcal{T}$ into two sub-collections:
$$\mathcal{T}_1:=\{t\in\mathcal{T}: 
2^{t}\psi_{1}(2^{t})\cdots \psi_n(2^{t})<1\}\qquad\text{and}\qquad 
\mathcal{T}_2=\mathcal{T}\setminus \mathcal{T}_1.
$$

Define $$J:=\left\{(q,\ba)\in\N\times \Z^d ~|~  \frac{\ba}{q}\in \bU, \left\vert qf_{j}\left (\frac{\ba}{q}\right)-b_j\right\vert <\frac{1}{2}\psi_{d+j}(q) \text{ for some } b_j\in \Z, 1\leq j\leq m\right\}.$$
Also, we define two sub-collections of $J$,
$(J_{i})_{i=1}^2$, that correspond to $\mathcal{T}_i$. We say that
$(q,\ba)\in J_{i}$ if $2^{t-1}\leq q\leq 2^t$ for some $t\in \mathcal{T}_{i}.$ 
Observe that 
$$J_1\cup J_2\subset J.$$
For $\alpha=(q,\ba)\in J$, define $\beta_{\alpha}:=q$ and $\mathcal{R}_{\alpha}:=\vv a/q$.
Also given $\bPsi=(\Psi_i)_{i=1}^d:\Rp^d\to\Rp$ let 
$$\Lambda^{\tau}_{\mathcal{R}}(\bPsi)=\{\bx\in U~|~\bx\in \tDelta(\mathcal{R}_{\alpha}, \bPsi(\beta_{\alpha})) \text{ for infinitely many }\alpha\in J_{\tau}\}\qquad( \tau=1,2),
$$ 
where $\tDelta$ is given by \eqref{Delta1}.

\begin{lemma}\label{lem: ubiqui is enough}
Let $\Psi_i(q)=\frac{\psi_i(q)}{2q}$ $(1\leq i\leq d)$, and $\beta, J$ as before. Then \begin{equation}\label{eq: enough}\Lambda_{\mathcal{R}}^{1}(\bPsi)\cup \Lambda_{\mathcal{R}}^{2}(\bPsi)\subset F^{-1}(\mathcal{S}_n(\psi_1,\cdots, \psi_n)).\end{equation}
\end{lemma}
\begin{proof}
    In what follows, $\tau=1,2$. By the mean value theorem, we can assume $f_j$ is Lipschitz for every $1\leq j\leq m$ with Lipschitz constant $1$. Suppose $\bx=(x_1,\cdots,x_d)\in \Lambda_{\mathcal{R}}^\tau(\bPsi)$, then there are infinitely many $(q,\ba,\bb)\in \N\times \Z^{n+1}, (q,\ba)\in J_{\tau}$ such that
$$
\begin{aligned}
&\vert x_i-a_i/q\vert<\frac{\psi_i(q)}{2q}&&&(1\le i\le d),\\
&\vert qf_j(\ba/q)-b_{j}\vert<\frac{1}{2}\psi_{d+j}(q)&&& (1\le j\le m).
\end{aligned}
$$ 
Then for $1\leq j\leq m,$
    
    $$\begin{aligned}
        \vert f_{j}(\bx)-b_j/q\vert&\leq\left\Vert \bx-\frac{\ba}{q}\right\Vert+\frac{\psi_{d+j}(q)}{2q}\\
        &\leq\max_{i=1}^d\frac{\psi_i(q)}{2q}+ \frac{1}{2q}\psi_{d+j}(q)&\stackrel{\eqref{eq: chain}}{\leq} \frac{\psi_{d+j}(q)}{q}.
    \end{aligned}$$
Since $\mathcal{T}$ is infinite, the proof is complete. 
\end{proof}

\subsection{Ubiquity and the finale}\label{final part}

From Lemma \ref{lem: ubiqui is enough}, it is enough to show that one of the sets on the left side of \eqref{eq: enough} has \textit{full} measure. In what follows $\brho=(\rho_i)_{i=1}^d\Rp^d\to\Rp$.
By the definition of $\Delta$ in \eqref{def: Delta},
$$\begin{aligned}\Delta(2^t, \tfrac12(\psi_{d+j}(2^t))_{j=1}^m, B, \brho(2^t)):=\bigcup_{(q,\ba,\bb)\in \mathcal{R}(2^t, (\psi_{d+j}(2^t))_{j=1}^m, B)}  \tDelta(\vv a/q,\brho(2^t)),
\end{aligned}
$$ 
for any ball $B\subset B_0,$ where $B_0$ be as in Proposition \ref{thm: ubiq main}. 
Now note that $$\begin{aligned}
       \bigcup_{(q,\ba,\bb)\in \mathcal{R}(2^t, (\psi_{d+j}(2^t))_{j=1}^m, B)} \tDelta(\vv a/q,\brho(2^t))\cap B
   \subset \bigcup_{2^{t-1}\leq q\leq 2^{t}, (q,\ba)\in J }\tDelta(\vv a/q,\brho(2^t))\cap B.\end{aligned}$$
By \eqref{eq3.22}, there is $i\in\{1,2\}$ such that
\begin{equation}\label{eq3.25}
\sum_{t\in \mathcal{T}_i} 2^t \psi_1(2^t)\cdots\psi_n(2^t)=\infty,
\end{equation}

\bigskip

\noindent\textbf{Case $1$:} \eqref{eq3.25} holds for $i=1$. Define 
$$
\rho_1(2^t):= \frac{\left((\psi_{2}\cdots\psi_n)(2^t) 2^t\right)^{-1}}{2^t}, \qquad \rho_i(2^t):=\frac{\psi_i(2^t)}{2^t}\quad (2\leq i\leq d)\qquad\text{for } t\in \mathcal{T}_1.
$$
For $t\in \N\setminus \mathcal{T}_1$ we set $\rho_i(2^t):=\frac{1}{2^t}.$ Now in the view \eqref{eq: standing 2 in 2^t} and \eqref{eq: psi chain}, for $t\in \mathcal{T}_1$ both \eqref{eq: standing 3} and \eqref{eq: chain} are satisfied by $(\psi_{i}(2^t))_{i=1}^{n}$, and $\rho_i(2^t)= \frac{\ep_i'}{2^t}$, where $\ep_i=\psi_i(2^t)$ $(1\leq i\leq n)$ and $\ep_i'$ as in \eqref{eq: simple ep'}. Thus, applying Proposition \ref{thm: ubiq main}, we get for $t\in \mathcal{T}_1,$
$$\mu_{d}\left(B \cap \bigcup_{2^{t-1}\leq q\leq 2^{t}, (q,\ba)\in J } \prod_{i=1}^dB\left(a_i/q, \rho_0\rho_i(2^t)\right)\right) \geq \frac{1}{2}\mu_{d}(B),$$ where $\rho_0>0$ be as in Proposition \ref{thm: ubiq main}. This shows that the system $\{\mathcal{R}_{\alpha};\beta\}$ is ubiquitous with respect to $\brho, \{u_t\}$ with $\{u_t\}$ being $\{2^t~|~t\in \mathcal{T}_1\}.$

By \eqref{eq: standing 2 in 2^t}, we have $\lim_{t\to \infty}\brho(t)=\vv0.$ Also by Remark~\ref{rem:goal 1}, for $t\in \mathcal{T}_1$, 
$$\rho_i(2^t)=\frac{\ep_i'}{2^t}\geq \frac{\ep_i}{2^t}=\frac{\psi_i(2^t)}{2^t}\geq \Psi_i(2^t).$$ 
Now since 
$$\sum_{t\in \mathcal{T}_1}\prod_{i=1}^d\frac{\Psi_i(2^t)}{\rho_i(2^t)}\stackrel{\eqref{eq: ep'}}{=} \sum_{t\in \mathcal{T}_1}2^t \psi_1(2^t)\cdots\psi_n(2^t) =\infty,$$ the required conclusion follows from Theorem KW in \S \ref{sec: abstract}.

\bigskip

\noindent\textbf{Case 2:} \eqref{eq3.25} holds for $i=2$. Then $\mathcal{T}_2$ is infinite. By definition, for any $t\in \mathcal{T}_2$,  
$$2^t\psi_1(2^t)\cdots \psi_{n}(2^t)>1.$$ Thus using Minkowski's convex body theorem, for any $\bx\in \bU$, there exists integer solution $(q, a_1,\cdots, a_{d}, b_1,\cdots, b_{m})\in \Z^{n+1}$ such that 
$$\begin{aligned}
    &\vert qx_i-a_i\vert \leq \psi_i(2^t)&&& (1\leq i\leq d), \\
    &\vert qf_{j}(\bx)-b_j\vert <\psi_{d+j}(2^t)&&& (1\leq j\leq m),\\
    & q\leq 2^t.
\end{aligned}$$
and we again arrive at the required conclusion.

\section{Weighted convergence: the special case}\label{sec: convergence}

The purpose of this section is to prove the convergence case of Theorem~\ref{thm: dream} for nondegenerate maps of the form \eqref{eq: F}. Let $F, f$ and $\bU$ be as in \eqref{eq: F}.
For $\bve=(\ep_i)_{i=1}^n$ with $0<\ep_i<1$, $t\in\N$, any $\Delta\subset \R^d$, let $R(\Delta,\bve,t)$ be the set of $(p_1,\dots,p_n,q)\in \Z^n\times \N$ such that
\begin{equation}\label{eq4.1x}
    \begin{aligned}
    &\left\vert f_{j}(\bx)-\frac{p_{d+j}}{q}\right\vert <\frac{\ep_{d+j}}{e^t}&&& (1\le j\le m),\\
    & \left\vert x_i-\frac{p_i}{q}\right\vert<\frac{\ep_i}{e^t} &&& (1\le i\le d),\\
    & 1\leq q\leq e^t.
    \end{aligned}
\end{equation}
for some $\bx\in\Delta$.
We also introduce the following counting function 
$$
N(\Delta,\bve,t):=\#R(\Delta,\bve,t).
$$

\begin{lemma}\label{lem: x to x_0}
Let $\bve=(\ep_i)_{i=1}^n$ satisfy \eqref{eq: chain}, $\bve'=(\ep_i')_{i=1}^d$ satisfy $0<\ep_i'<1$, $t\in \N$ and $\bx_0=(x_{0,i})_{i=1}^d\in \bU$. Suppose that 
\begin{equation}\label{eq4.2x}
    \bx\in \Delta=\Delta(\bx_0):=\prod_{i=1}^d B\left(x_{0,i}, \left(\frac{\ep_i'}{e^t}\right)^{1/2}\right)\subset \bU
\end{equation}
    and $(\bp,q)\in\Z^{n+1}$ satisfy \eqref{eq4.1x}.
Then 
$$\begin{aligned}
&\left\vert qf_{j}(\bx_0)-\sum_{i=1}^d \partial_i f_{j}(\bx_0)(qx_{0,i}-p_i)-p_{d+j}\right\vert \ll_{F, n}\max\{\ep_{d+j}, \max_{i=1}^d\ep_i'\}&&& (1\le j\le m),\\
    & \left\vert qx_{0,i}-p_i\right\vert<2\max\{\ep_i, (\ep_i' e^t)^{1/2}\}&&& (1\le i\le d),\\
    & 1\leq q\leq e^t.
    \end{aligned}$$
\end{lemma}
\begin{proof}
Using the hypothesis of this lemma, we obtain that 
$$
\begin{aligned}
&\left\vert qf_{j}(\bx)-\sum_{i=1}^d \partial_i f_{j}(\bx_0)(qx_{i}-p_i)-p_{d+j}\right\vert <\max\{\ep_{d+j}, \max_{i=1}^d\ep_i\}\stackrel{\eqref{eq: chain}}{=}\ep_{d+j}&&& (1\le j\le m),\\
    & \left\vert qx_{i}-p_i\right\vert<\ep_i&&& (1\le i\le d),\\
    & 1\leq q\leq e^t.
    \end{aligned}$$
By \eqref{eq4.1x}, for each  $1\leq i\leq d$, $x_i=x_{0,i}+x_i'\left(\frac{\ep_i'}{e^t}\right)^{1/2}$ with  $\vert x_i'\vert \leq 1$. Then 
$$\begin{aligned}\vert q x_{0,i}-p_i\vert<&\vert q x_{i}-p_i\vert+ \vert q\vert \vert x_i-x_{0,i}\vert\\
    &\leq \varepsilon_i+ e^t \left(\frac{\ep_i'}{e^t}\right)^{1/2}\leq 2\max\{\ep_i, (\ep_i' e^t)^{1/2}\}.\end{aligned}$$
For $1\leq j\leq m,$
$$\begin{aligned}
&\left\vert q \fff_j(\bx_0)-\sum_{i=1}^d\partial_i\fff_j(\bx_0)(q x_{0,i}-p_i)-p_{d+j} \right\vert\\
& \ll_{F,n}\varepsilon_{d+j}+\sum_{i=1}^d\vert \partial_i\fff_j(\bx_0)-\partial_i\fff_j(\bx) \vert \vert qx_i-p_i\vert +\vert q\vert \max_{i=1}^d\left(\varepsilon'_i e^{-t} \right)\\
& \stackrel{\eqref{eq: chain}}{\ll_{F,n}}\varepsilon_{d+j}+ \max_{i=1}^d(\ep_i' e^{-t})^{1/2}\varepsilon_{d}+ \max_{i=1}^d\ep_i'\\
& \stackrel{\eqref{eq: chain}}{\ll_{F,n}}\max\{\ep_{d+j}, \max_{i=1}^d \ep_i'\}.
\end{aligned}
$$    
\end{proof}

Let $\bve=(\ep_i)_{i=1}^n$ and $\bve'=(\ep_{i}')_{i=1}^d$ be as in Lemma \ref{lem: x to x_0}. Define the following matrices:
\begin{align}
   & g_{conv}:= \phi\diag(\ep_n^{-1},\cdots,\ep_1^{-1}, e^{-t});\label{eq4.1}\\
   & a:=\diag(\underbrace{1,\cdots, 1}_{m}, \frac{1}{\ep_d^{-1} \max\{\ep_d, (\ep_d'e^t)^{1/2}\}},\cdots, \frac{1}{\ep_1^{-1} \max\{\ep_1, (\ep_1'e^t)^{1/2}\}}, 1),\label{eq4.2}
\end{align}
where 
\begin{equation}\label{eq: phi}    \phi^{n+1}:=e^t\prod_{i=1}^n \ep_i .
\end{equation}

\subsection{Choosing $\ep_i'$}

Given $\bve=(\ep_i)_{i=1}^n$ and $t\in \N$, we choose $\bve'=(\ep_{i}')_{i=1}^d$ so that
\begin{equation}\label{ep'<epd+1}
    \max_{i=1}^d \ep_i'\leq \ep_{d+1},
\end{equation}
and
\begin{equation}\label{eqn: ep' condition 2}
    \ep_i\leq (\ep_i'e^t)^{1/2}\qquad\left(\text{or equivalently } \frac{\ep_i}{e^t}\leq \left(\frac{\ep_i'}{e^t}\right)^{1/2}\right)\qquad\text{for $1\leq i\leq d$}.
\end{equation}
Define 
\begin{equation}\label{eq4.8}
\mathfrak{M}=\mathfrak{M}(\bve,\bve', t):=\{\bx\in \bU~|~\lambda_{n+1}(ag_{conv} u_1(\bx)\Z^{n+1})>\phi\}.
\end{equation}

\begin{lemma}\label{lem: counting locally on good}
Let $B$ be a ball in $\bU$. Suppose that $\bve=(\ep_i)_{i=1}^n$ satisfy \eqref{eq: chain}, $t\in\N$, and $\bve'=(\ep_i')_{i=1}^d\in(0,1)^d$ satisfy \eqref{ep'<epd+1} and \eqref{eqn: ep' condition 2}. Let $\Delta$ be given by \eqref{eq4.2x} and $\mathfrak{M}$ be given by \eqref{eq4.8}. Then for every $\bx_0$ in $B\cap \bU\setminus \mathfrak{M}$,

    $$
    N\left(\Delta\cap B, \bve, t\right)~~\ll_{F,n}~~ \ep_{d+1}\cdots\ep_n e^{t(d+1)}\left(\frac{\ep_1'\cdots\ep_d'}{e^{dt}}\right)^{1/2}.
    $$
\end{lemma}

\begin{proof}
Let $(\bp, q)\in \Z^{n+1}$ be an integer vector that contributes to the left side above. Then by Lemma \ref{lem: x to x_0}, \eqref{ep'<epd+1} and \eqref{eqn: ep' condition 2}, 
$$
\begin{aligned}
&\left\vert qf_{j}(\bx_0)-\sum_{i=1}^d \partial_i f_{j}(\bx_0)(qx_{0,i}-p_i)-p_{d+j}\right\vert \ll_{F,n}\ep_{d+j}&&& (1\le j\le m),\\
    & \left\vert qx_{0,i}-p_i\right\vert<2(\ep_i' e^t)^{1/2}&&& (1\le i\le d),\\
    & 1\leq q\leq e^t.
    \end{aligned}$$

This implies 
$$g_{conv}u_1(\bx_0)(-\bp,q)\in [c_1\phi]^{m}\times \prod_{i=1}^d [c_1\phi\ep_i^{-1} (\ep_i'e^t)^{1/2}]\times [c_1\phi],$$ where $\phi$ is an in \eqref{eq: phi}, $c_1>0$ depends on $n, F$ only and $[x]$ denotes the interval $[-x,x]$.
By \eqref{eqn: ep' condition 2} this implies,
$$ag_{conv}u_1(\bx_0)(-\bp,q)\in  [c_2\phi]^{n+1},$$ where $c_2>0$ only depends on $n,F.$ 
Since $\bx_0\notin\mathfrak{M}$, we get that
$$
N\left(\Delta\cap B, \bve, t\right) \ll_{n,F} \phi^{n+1} \prod_{i=1}^d\frac{\max\{\ep_i, (\ep_i' e^t)^{1/2}\}}{\ep_i}.$$
By \eqref{eq: phi} and \eqref{eqn: ep' condition 2}, this is further bounded above by
$$
\ep_{d+1}\cdots\ep_n e^{t(d+1)} \left(\frac{\ep_1'\cdots\ep_d'}{e^{dt}}\right)^{1/2}
$$
as claimed.    
\end{proof}

Next, we state a lemma analogous to \cite[Lemma 5.4]{BY} that follows from the covering argument and the fact that 
$\frac{\ep_i'}{e^t}\leq \frac{1}{e^t}\to 0$ as $t\to\infty.$

\begin{lemma} Let $B$ be a ball in $\bU$. Then for sufficiently large  $t\in\N$ and every choice of $\bve$ and $\bve'$ as in Lemma \ref{lem: counting locally on good}, we have that
$$
N(B\setminus \mathfrak{M},  \bve, t) \le \left(\frac{\ep_1'\cdots\ep_d'}{e^{dt}}\right)^{-1/2} \max_{\bx_0\in B\setminus \mathfrak{M}} N\left(\Delta(\bx_0)\cap B,\bve, t\right) \mu_d(B),$$
where $\mathfrak{M}$ is given by \eqref{eq4.8} and $\Delta(\bx_0)$ is given by \eqref{eq4.2x}.
\end{lemma}

The above two lemmas immediately give the following proposition.
\begin{proposition}\label{prop: count on major}
     Let $B$ be a ball in $\bU$. Then for sufficiently large  $t\in\N$ and every choice of $\bve$ and $\bve'$ as in Lemma \ref{lem: counting locally on good}, we have that
$$
N(B\setminus \mathfrak{M},  \bve, t) \ll \ep_{d+1}\cdots\ep_n e^{t(d+1)} \mu_d(B),
$$
where $\mathfrak{M}$ is given by \eqref{eq4.8}.
\end{proposition}

\subsection{Dealing with $\mathfrak{M}$}\label{sec: minor is small}

As before, throughout this subsection, $t\in\N$, $\bve=(\ep_i)_{i=1}^n$ satisfies \eqref{eq: chain}, $\bve'=(\ep_i')_{i=1}^d\in(0,1)^d$ satisfies \eqref{ep'<epd+1} and \eqref{eqn: ep' condition 2}, and $\mathfrak{M}=\mathfrak{M}(\bve,\bve', t)$ is given by \eqref{eq4.8}.
Suppose $\bx_0\in \mathfrak{M}$. By definition,
$$\lambda_{n+1}(ag_{conv} u_1(\bx_0)\Z^{n+1})>\phi.$$ By Lemma \ref{dual lemma}, we have that 
\begin{equation}\label{eq: lambda1}
\lambda_1(a^\star g_{conv}^\star u_1^\star(\bx_0)\Z^{n+1})\ll_{n}\phi^{-1}.\end{equation}
Using \eqref{eqn: ep' condition 2}, we observe  that 
$$\begin{aligned} &g_{conv}^\star=\phi^{-1}\diag(e^t,\ep_1,\cdots,\ep_n),\\
& a^\star=\diag(1, \left(\frac{\ep_1'}{e^t}\right)^{1/2} \frac{e^t}{\ep_1}, \cdots, \left(\frac{\ep_d'}{e^t}\right)^{1/2} \frac{e^t}{\ep_d}, \underbrace{1,\cdots, 1}_{m}).\end{aligned}
$$
Also, for $\bx\in \bU$,
$u_1^\star(\bx)$ is given by 
\begin{equation}\label{eqn:dual u_1(x)}
        u_1^\star(\bx)=\begin{bmatrix}
            1 & -\bx & -f(\bx)\\
            0 & \mathrm{I}_d & J(\bx)\\
            0 & 0   & \mathrm{I}_m
        \end{bmatrix}\,,
    \end{equation} 
where 
$$
J(\bx)=\big[\partial_i\fff_j(\bx)\big]_{1\leq i\leq d,\, 1\leq j \leq m}\, 
$$ 
is the Jacobian of the map $\bx\mapsto f(\bx)=(f_1(\bx),\dots,f_m(\bx))$.
Then, by \eqref{eq: lambda1}, there exists $(a_0,\ba)\in\Z^{n+1}$ such that 
$$\begin{aligned}
    & \vert a_0+\ba\cdot F(\bx_0)\vert<e^{-t},\\
    & \vert\partial_iF(\bx_0)\cdot\ba\vert< (\ep_i' e^t)^{-1/2}, ~1\leq i\leq d,\\
    & \vert a_{d+j}\vert<\ep_{d+j}^{-1}, ~1\leq j \leq m.
\end{aligned}
$$
From the last two inequalities and from $F(\bx)=(\bx,f(\bx))$ and \eqref{eq: chain}, we get that 
$$\vert a_i\vert \leq (\ep_i' e^t)^{-1/2}+\ep_{d+1}^{-1}\qquad (1\leq i\leq d).$$
Therefore, with $\mathfrak{S}_F(\cdots)$ as in \eqref{eq: new set}, we have that
$$
\mathfrak{M}\subset \mathfrak{S}_{F}(\delta,K,\dots,K,\bT)
$$ 
with the following choices of $\delta$, $\bK=(K_i)_{i=1}^d$ and $\bT=(T_i)_{i=1}^n$:
 $$\begin{aligned}
&\delta=e^{-t},\\ 
& K=\max\{K_1,\dots,K_d\},\\
&K_i= (\ep'_i e^t)^{-1/2}&&(1\le i\le d), \\ 
& T_i=K_i+\ep_{d+1}^{-1}&& (1\leq i\leq d),\\
& T_{d+j}=\ep_{d+j}^{-1}&& (1\leq j\leq m).
\end{aligned}
$$
Note that 
\begin{equation}\label{eq: K is smaller than T}
    K_i\leq \ep_{d+1}^{-1}\iff \frac{\ep_{d+1}}{e^t}\leq \left(\frac{\ep_i'}{e^t}\right)^{1/2}\qquad (1\leq i\leq d).
\end{equation}

\subsubsection{Choice of $\ep_i'$}
We choose \begin{equation}\label{eq: choice of ep'}
    \ep_i'=\ep_{d+1}\qquad (1\leq i\leq d).
\end{equation}
Then \eqref{ep'<epd+1} is satisfied trivially, by \eqref{eq: chain} we have \eqref{eqn: ep' condition 2} and \eqref{eq: K is smaller than T} follows since $\ep_{d+1}<1$.
To summarise, now we have,
$$
\begin{aligned}
& K_i=  (\ep_{d+1}e^t)^{-1/2}&& (1\leq i\leq d),\\
& T_i=\ep_{d+1}^{-1}&& (1\leq i\leq d),\\
& T_{d+j}=\ep_{d+j}^{-1}&& (1\leq j\leq m).
\end{aligned}
$$
Then using the measure estimates as in \cite[Theorem 1.4]{BKM}, we get the following:

\begin{proposition}\label{Bound on minor}
    Let $F:\bU\to\R^n$ be a map as in \eqref{eq: F}, where $\bU$ be an open set in $\R^d$. Let $F$ be $l$-nondegenerate at $\bx_0\in \bU.$ Then there exists a ball $B_0$ centered at $\bx_0$, constant $t_0$ depending on $n,F, B_0$ such that  for $t\geq t_0,$ and any $\bve=(\ep_i)_{i=1}^n$ satisfying \eqref{eq: chain},
$$
\mu_d(\mathfrak{M}\cap B_0) \ll_{F, B_0} \left(\frac{\ep_{d+1}^{1/2}}{e^{t/2}} (e^{-t}\ep_{d+1}^{-d}\ep_{d+1}^{-1}\cdots\ep_{n}^{-1})\right)^{\alpha},
    $$ 
where $\mathfrak{M}=\mathfrak{M}(\bve,\bve', t)$ is given by \eqref{eq4.8} with $\bve'$ defined by \eqref{eq: choice of ep'} and $\alpha= \frac{1}{d(2l-1)(n+1)}.$
\end{proposition}

\subsection{Finishing the proof}\label{sec: final convergence}

Since $F$ is $l$-nondegenerate at $\bx_0$ then it is enough to show

$$
\mu_d(F^{-1}\mathcal{S}_n(\psi_1,\cdots,\psi_n)\cap B_0)=0,
$$
when $\sum_{q}\psi_1(q)\cdots\psi_n(q)<\infty.$ Then for all $t\geq T, T\geq 1$ such that 

$$F^{-1}\mathcal{S}_n(\psi_1,\cdots,\psi_n)\cap B_0\subset \bigcup_{t\geq T}\underbrace{\mathfrak{M}((e\psi_i(e^{t-1}))_{i=1}^n, t)\cap B_0}_{A_t}\cup \bigcup_{t\geq T} B_t, $$ where 
$$
B_t:= \bigcup_{(q, p_1,\cdots, p_n)\in R(B_0\setminus\mathfrak{M}((e\psi_i(e^{t-1}))_{i=1}^n, t),(\psi_i(e^{t-1}))_{i=1}^n,t-1)} \prod_{i=1}^d \left(\frac{p_i}{q}-\frac{\psi_i(e^{t-1})}{e^{t-1}}, \frac{p_i}{q}+\frac{\psi_i(e^{t-1})}{e^{t-1}}\right)\cap B_0.
$$
By \eqref{eq: prod is greater than}, and \eqref{eq: psi chain} we get that 
$\psi_{d+1}^d(e^{t-1})\psi_{d+1}(e^{t-1})\cdots \psi_{n}(e^{t-1}) e^{t-1}>e^{-\mathfrak{c}(t-1)}.$
Thus by
Proposition \ref{Bound on minor}, we get that 
$$\mu_d(A_t)\leq \left(\frac{e^{\mathfrak{c}t}}{e^{t/2}}\right)^{\alpha}.$$ Since $\mathfrak{c}<\frac{1}{2},$ we get 
$\sum_{t\geq T} \mu_d(A_t)<\infty.$ Now by Proposition \ref{prop: count on major}, we get that 
$$\mu_d(B_t)\leq \prod_{i=1}^d \frac{\psi_i(e^{t-1})}{e^{t}} e^{t(d+1)}\prod_{j=1}^m \psi_{d+j}(e^{t-1})\mu_{d}(B_0)=e^t \prod_{i=1}^n\psi_{i}(e^{t-1}) \mu_{d}(B_0).$$ By monotonicity of each $\psi_i$ and convergence of the series $\sum \psi_1(q)\cdots \psi_n(q),$ we get that 
$\sum_{t\geq T} \mu(B_t)<\infty.$ By the convergence Borel-Cantelli lemma, the proof is complete.

\section{The general case}\label{sec: general}

In this section, we point out the modifications required to extend the previous arguments from the case \eqref{eq: F}, namely $F=(\bx,f(\bx))$, to arbitrary nondegenerate maps. Since we impose no conditions on $\mathcal{M}$ beyond nondegeneracy, and since we approximate each coordinate by different functions, we must deal with manifolds that are not of the form \eqref{eq: F}. For instance, a nondegenerate manifold may contain a subset of positive measure on which the tangent plane is parallel to one of the coordinate axes. This occurs, for example, for the cylinder
$$
\{(x,\sqrt{1-x^2}, z) : x\in(0,1),\ z\in\R\},
$$
which cannot be written in the form \eqref{eq: F} with all independent variables appearing first.

Suppose $F:\bU\to\R^n$ is nondegenerate at almost every point, where $F(\bx)=(F_{i}(\bx))_{i=1}^n, F_i:\bU\to \R$. Then, for almost every point $\bx_0\in\bU$ we have that
$(\partial_i F_j(\bx_0))_{i=1}^d \neq \mathbf{0}$ for every $1\le j\le n$, and
\begin{equation}\label{eq5.1}
\operatorname{rank}
\begin{bmatrix}
\partial_1 F_1(\bx_0) & \cdots & \partial_1 F_n(\bx_0) \\
\vdots & \ddots & \vdots \\
\partial_d F_1(\bx_0) & \cdots & \partial_d F_n(\bx_0)
\end{bmatrix}
= d.
\end{equation}
This can be seen as a consequence of the nondegeneracy of $F$ on $\bU$ by using repeated differentiation as in \cite[Lemmas~2~and~3]{MR1387861}, together with an application of Fubini's theorem to extend the argument to functions of several variables.

Let $\bx_0\in\bU$ be as above. Let $\mathcal{A}$ be the collection of all $d$-tuples
$(i_1,\ldots,i_d)$ with $1\le i_1<\cdots<i_d\le n$ such that the submatrix of \eqref{eq5.1} with columns $\{i_j\}_{j=1}^d$ has rank $d$.
Since $(\partial_i F_1(\bx_0))_{i=1}^d \neq \mathbf{0}$, we have that
$
\min\{\, i_1 : (i_j)_{j=1}^d \in \mathcal{A} \,\} = 1.
$
By the Implicit Function Theorem, we can write
$$
F(\bx) = (x_1,\tilde{F}_2(\bx),\ldots,\tilde{F}_n(\bx))
$$
on a sufficiently small neighborhood $V_0$ of $\bx_0$. Now, by nondegeneracy, for almost every point $\bx_1\in V_0$, we have two possibilities:
\begin{itemize}
\item $\tilde{F}_2(\bx)$ is a function of $x_1$ alone on a neighborhood $V_1$ of $\bx_1$, in which case, on $V_1$,
$$
F(\bx) = (x_1,g_1(x_1),\tilde{F}_3(\bx),\ldots,\tilde{F}_n(\bx));
$$
\item applying another change of variables on a sufficiently small neighborhood $V_1$ of $\bx_1$, we can write $F$ in the form
$$
F(\bx) = (x_1,x_2,\tilde{\tilde{F}}_3(\bx),\ldots,\tilde{\tilde{F}}_n(\bx)).
$$
\end{itemize}

By continuing this procedure inductively on the remaining coordinate functions, we can write $F$ on a neighborhood $V$ of almost every point in the form
\begin{equation}\label{eq: Para_new}
(\bx_1,g_1(\bx_1),\bx_2,g_2(\bx_1,\bx_2),\ldots,\bx_{\mathfrak{s}},g_{\mathfrak{s}}(\bx)), 
\qquad \bx\in V\subset\R^d,
\end{equation}
where $\bx_k\in\R^{d_k}$ and $g_k(\bx_1,\ldots,\bx_k)\in\R^{m_k}$ for $1\le k\le\mathfrak{s}$,
$d_1+\cdots+d_{\mathfrak{s}}=d$, $m_1+\cdots+m_{\mathfrak{s}}=m$, and each $g_k$ depends only on $\bx_1,\ldots,\bx_k$.

\subsection{Further notation}\label{sec: notation general}  

Let 
$$
\{1,\ldots,n\} = I(1) \cup J(1) \cup \cdots \cup I(\fs) \cup J(\fs),
$$ 
where for $1 \le k \le \fs$, $I(k)$ and $J(k)$ are strictly increasing consecutive subsets of $\N$ with 
$$
\# I(k) = d_k, \quad \# J(k) = m_k, \quad \sum_{i=1}^{\fs} d_i = d, \quad \sum_{i=1}^{\fs} m_i = m.
$$
For convenience, we define 
$$
I = \bigcup_{k=1}^{\fs} I(k), \quad 
J = \bigcup_{k=1}^{\fs} J(k), \quad
\bar I_k = \bigcup_{i=1}^k I(i), \quad
\bar J_k = \bigcup_{j=k}^{\fs} J(j), \quad 1 \le k \le \fs.
$$

We assume that every number in $J(k)$ is greater than every number in $I(k)$, and smaller than every number in $I(k+1)$. Moreover, $I(k)$, $J(k)$ for $1\le k \le \fs-1$ and $I(\fs)$ are nonempty, but $J(\fs)$ may be empty.

For convenience, we introduce the notation 
$$
\bar\bx_k := (\bx_1,\ldots,\bx_k) \in \R^{d_1+\cdots+d_k}, \quad 1 \le k \le \fs.
$$ 
Note that $\bar\bx_{\fs} = \bx \in \R^d$ and that 
$$
\bx = (x_i)_{i\in I} \in \R^d, \qquad \mathbf{y} = (y_k)_{k\in I\cup J} \in \R^n.
$$
With this notation, we can rewrite \eqref{eq: Para_new} as
$$
(\bx_1, g_1(\bx_1), \bx_2, g_2(\bar\bx_2), \ldots, \bx_{\fs}, g_{\fs}(\bar\bx_{\fs})).
$$

Define the matrix $U_1(\bx)$ by
\begin{equation}\label{def_U_1}
\begin{bmatrix}
1_{m_{\fs}}\!\! & -\nabla_{\bx_{\fs}} g_{\fs} & 0_{m_{\fs} \times m_{\fs-1}} & -\nabla_{\bx_{\fs-1}} g_{\fs} & \cdots & -\nabla_{\bx_1} g_{\fs} & g_{\fs} - \sum_{i=1}^{\fs} \bx_i^T \nabla_{\bx_i} g_{\fs} \\
& 1_{d_{\fs}} & & & & & \bx_{\fs} \\
& & 1_{m_{\fs-1}} & -\nabla_{\bx_{\fs-1}} g_{\fs-1} & \cdots & -\nabla_{\bx_1} g_{\fs-1} & g_{\fs-1} - \sum_{i=1}^{\fs-1} \bx_i^T \nabla_{\bx_i} g_{\fs-1} \\
& & \vdots & \vdots & \vdots & \vdots & \vdots \\
& & & & 1_{m_1} & -\nabla_{\bx_1} g_1 & g_1 - \bx_1^T \nabla_{\bx_1} g_1 \\
& & & & & 1_{d_1} & \bx_1 \\
& & & & & & 1
\end{bmatrix},
\end{equation}
where $1_k$ is the identity $k\times k$ matrix and $0_{k\times m}$ is the zero $k\times m$ matrix.
Note that in \eqref{def_U_1}, the first $m_{\fs}$ rows are evaluated at $\bar\bx_{\fs}$, the next $d_{\fs}$ rows are evaluated at $\bx_{\fs}$, the next $m_{\fs-1}$ rows are evaluated at $\bar\bx_{\fs-1}$, and so on.

Now, the dual matrix is given by
$$U_1^\star(\bx)= \begin{bmatrix}
    1 & -\bx_1 & -g_1(\bx_1) & -\bx_2 & -g_2(\bar\bx_2) & \cdots & -g_{\fs}(\bx)\\
    & 1_{d_1} & \nabla_{\bx_1}g_1 & 0_{d_1\times d_2} & \nabla_{\bx_1}g_2 & \cdots & \nabla_{\bx_1}g_{\fs}\\
    & & 1_{m_1} & & & & \\
    & & & 1_{d_2} & \nabla_{\bx_2}g_2 & \cdots & \nabla_{\bx_2}g_{\fs}\\
    & & &\vdots & \vdots &\vdots  &\vdots\\
    & & & & &  1_{d_{\fs}} & \nabla_{\bx_{\fs}}g_{\fs}\\
    & & & & & & 1_{m_{\fs}}
\end{bmatrix}.
$$
Note there exists two permutation matrices  $\omega_1, \omega_2\in \mathrm{GL}_{n+1}(\R)$ such that $$U_1^\star(\bx)= \omega_1 u_1^\star(\bx)\omega_2,$$  where $u_1^\star(\bx)$ is as in \eqref{eqn:dual u_1(x)} for the map 
$$
\bx\to (\underbrace{\bx_1,\cdots,\bx_{\fs}}_{\bx}, g_1(\bx_1),\cdots, g_{\fs}(\bx))=  F(\bx)\omega_2:=\tF(\bx)$$ 
with $\bx_i$ having the same meaning as  in \S \ref{sec: notation general}.

\subsection{Divergence}
To generalise the proof of the divergence case of Theorem~\ref{thm: dream} from the special $F$ given by \eqref{eq: F} to general $F$ we only need to verify that Proposition~\ref{thm: ubiq main} holds for general $F$ as in \eqref{eq: Para_new}. Let $c>0$, $\ep_i>0$ ($1\leq i\leq n$) satisfy \eqref{eq: chain} and $Q>0$. Define 
\begin{equation}\label{eq: new g}
    g:=c^{-1}\diag(\underbrace{\ep_{i}}_{i\in J(\fs)}, \underbrace{\ep_i'}_{i\in I(\fs)}, \cdots, \underbrace{\ep_{i}}_{i\in J(1)},\underbrace{\ep'_{i}}_{i\in I(1)}, c^{n+1}Q),
\end{equation}
where $\ve'_i$ $(i\in I)$ are positive and we assume, replacing \eqref{eq: ep'}, that
\begin{equation}\label{eq: ep'+}
\left(\prod_{i\in I}\ve'_i\right)\left(\prod_{j\in J}\ve_j\right) Q=1.
\end{equation}
Define 
$$\ep_{I(k)}:=\prod_{i\in I(k)}\ep_i,\qquad \ep_{J(k)}=: \prod_{j\in J(k)}\ep_j\qquad\text{and}\qquad
\ep_{I(k)}'=\prod_{i\in I(k)}\ep_i'\quad (1\leq k\leq \fs).
$$ 
Then, with the notation introduced in this section, we can rewrite \eqref{eq: ep'+} as 
$$
\prod_{k=1}^{\fs}\ep_{I(k)}'\ep_{J(k)} Q=1.
$$
Note that 
 \begin{equation}\label{eq: new g dual}
    ( g^{-1})^\star= c^{-1}\diag(c^{n+1}Q, \underbrace{\ep'_{i}}_{i\in I(1)}, \underbrace{\ep_{j}}_{j\in J(1)}, \cdots,  \underbrace{\ep_i'}_{i\in I(\fs)}, \underbrace{\ep_{j}}_{j\in J(\fs)}).
 \end{equation} 
Let  
$$
\mathcal{G}=\mathcal{G}(c,Q, \{\ep_j, \ep_{i}'\}_{i\in I, j\in J}):=\{\bx\in \bU: \lambda_{n+1}(g^{-1} U_1(\bx)\Z^{n+1})\leq c^{-n}\}.
$$ 
This generalises \eqref{def: mathcal{G}} to the case of a general $F$, with $g$ redefined and $u(\bx)$ replaced by $U_1(\bx)$.

The following lemma is a generalisation of Lemma~\ref{lem: projection}, and its proof follows the same line of argument with only minor modifications. Let us recall from \S \ref{sec: projection}, that for any open set $V\subset \R^d$ and any $\br=(r_i)_{i=1}^d\in\Rp^d$, 
$$
V^{(\br)}=\{\bx\in \R^d~|~ \tDelta(\bx,\br)=\prod_{i=1}^d B(x_i, r_i)\subset V\}.
$$

\begin{lemma}\label{lem: new projection}
Let $Q>0$, $\varepsilon_i>0$ $(1\leq i\leq n)$ satisfy \eqref{eq: chain}, $\ep_i'>0$ $(i\in I)$, $c<1$. Suppose that \eqref{eq: ep'+} holds and  \begin{equation}\label{eq: new restriction linear}
\min_{j\in J(k)} \varepsilon_{j}>\frac{1}{Q}\max_{i, i' \in\bar I_k}\varepsilon_i' \varepsilon'_{i'}\qquad\text{for each } 1\leq k\leq \mathfrak{s}.
\end{equation}
For $i\in I$ let $r_i=\varepsilon'_i(cQ)^{-1}$. Then for any $\bx\in \mathcal{G}\cap \bU^{(\br)}$ there exists an integer point 
$$
(q,\bb)=\big(q, (b_i)_{i\in I(1)}, (b_j)_{j\in J(1)},\cdots, (b_i)_{i\in I(\mathfrak{s})}, (b_j)_{j\in J(\mathfrak{s})}\big)
$$
in
$$
\N\times \Z^{d_1}\times \Z^{m_1}\times \cdots\times \Z^{d_\mathfrak{s}}\times \Z^{m_\mathfrak{s}}=\N\times\Z^{n}
$$ 
such that for $1\leq k\leq \mathfrak{s}$\begin{equation}\label{eq: linear_div_1}
    \begin{aligned}
    &\vert qx_i-b_i\vert\ll_{n} \frac{n+1}{c^{n+1}} \varepsilon_i' \qquad \text{for }i\in I(k), \\
& \vert qg_{j}((b_i)_{i\in \bar I_k}/q)-b_j\vert\leq \frac{(n+1)M}{c^n}\ep_{j} \qquad\text{for } j\in J(k),\\[1ex]
& (n+1)Q\leq q\leq 3(n+1)Q,
    \end{aligned}     
    \end{equation}
    where $M$ is as in \eqref{eq: bound on derivatives}.
\end{lemma}

\begin{remark}
The importance of the structure in $F$ given by \eqref{eq: Para_new} manifests in \eqref{eq: new restriction linear}, since on the right side of the equation only $\bar{I_{k}}$ appears, instead of $I$.
\end{remark}

Now, armed with Lemma~\ref{lem: new projection}, we demonstrate how to complete the proof of the divergence case. Using Lemma~\ref{dual lemma} and \eqref{eq: chain}, we get that the complement to $\mathcal{G}$ is contained in
$$
\{\bx\in \bU~|~\lambda_1((g^{-1})^\star U_1(\bx)^\star\Z^{n+1})\ll_{n} c^n\}
$$ 
which is further a subset of 
\begin{equation}\label{eq: new BKM set_gen}\left\{\bx\in \bU:\exists\;(a_0,\ba)\in\Z\times\Z^n_{\neq\bf0}\;\;\text{such that }\left.
\begin{array}{l}
|a_0+F(\bx)\ba^\top|<Q^{-1}\\[1ex]
\vert \partial_iF(\bx)\ba^\top\vert<K_i\quad (i\in I)\\[1ex]
\vert a_k\vert<T_k\quad (k\in J) 
\end{array}
\right.\right\},
\end{equation} where $\ba=(a_i)_{i\in I,J}$,  $\delta$, $(K_i)_{i\in I}$, $T_1,\cdots, T_n$ are positive real parameters defined as follows: \begin{equation}\label{eq: KT general}
\begin{aligned}
&K_i\asymp_{n}c^{n+1}\varepsilon_i'^{-1}\qquad (i\in I), \\ 
& T_i \asymp_{n, F}c^{n+1}\varepsilon_i'^{-1}+ c^{n+1}\Big(\min_{j\in  J(k)}{\ep_j}\Big)^{-1}\qquad (i\in I(k),~1\leq k\leq \fs), \\
& T_{j} \asymp_{n}c^{n+1}\varepsilon_{j}^{-1}\qquad (j\in J).\end{aligned}
\end{equation}

\begin{remark}Note that by the definition of $\bar J_k$, and \eqref{eq: chain}
we get that $\min_{j\in \bar J_k}{\ep_j}= \min_{j\in J(k)}{\ep_j}$, which we used above in the second equation. Note that parametrization \eqref{eq: Para_new} is crucially used. 
\end{remark}

We choose $(\ep_i')_{i\in I}$ similarly to \eqref{eq: simple ep'} as follows:
\begin{equation}\label{eq: simple ep'_new}
    \ep_1'=\frac{1}{\ep_2\cdots \ep_n Q}\qquad\text{and}\qquad \ep_i'=\ep_i\quad (i\in I\setminus\{1\}).
\end{equation}

By \eqref{eq: KT general}, using \eqref{eq: chain} we get that up to constants depending on $F,n,m,$
$$\begin{aligned}
    & K_1\asymp c^{n+1}\ep_{1}'^{-1},\\
    & K_i\asymp c^{n+1}\ep_i^{-1}
    \qquad(i\in I\setminus\{1\}),\\
    & T_1\asymp nc^{n+1}\max\left\{\ep_{1}'^{-1}, \left(\min_{j\in J(1)} \ep_{j}\right)^{-1}\right\},\\
& T_i\asymp c^{n+1} n \ep_{i}^{-1}\qquad (2\leq i\leq n).
    \end{aligned}
    $$
With these choices, we can exactly follow the proof of Proposition \ref{prop: crucial in Case 2}. We observe that, in \eqref{eq: new BKM set_gen} the set with $\bx\to F(\bx)$ is the same as the set with $\bx\to \tF(\bx):=F(\bx)\omega_2$, where $\omega_2$ is a permutation matrix, and $\tF$ is in the form \eqref{eq: F}. The upshot is that, we can apply Theorem~\ref{BKM_new2} to extend Proposition \ref{prop: crucial in Case 2} to general $F$. Using Proposition \ref{prop: crucial in Case 2}, we derive Theorem~\ref{thm: ubiq main} for general $F$ as parametrized in \eqref{eq: Para_new} and it remains to apply Theorem~KW to finish.

\subsection{Convergence}
Here we generalise the proof of the convergence case of Theorem~\ref{thm: dream} from the special $F$ given by \eqref{eq: F} to general $F$.

Let $F$ be parametrized as in \eqref{eq: Para_new}, for every $\bx\in V\subset \bU.$
In what follows, we take \begin{equation}\label{eq: ep' new}
    \ep_{i}'=\min_{j\in J(k)}\ep_{j}\qquad (i\in I(k),~1\leq k\leq \fs).
\end{equation}
This choice implies that
\begin{equation}\label{eq:epep'new}
\ep_i\leq \ep_{i}'\qquad (i\in I(k),~1\le k\le \fs),\end{equation} which further gives that 
$\ep_{i}\le \left(\ep_i' e^t\right)^{1/2}$ for $i\in I(k)$, $1\le k\le \fs$.

The same line of proof as in Lemma \ref{lem: x to x_0} together with \eqref{eq:epep'new} gives the following lemma.
\begin{lemma}\label{lem: gen x to x_0}
Let $\bve=(\ep_i)_{i=1}^n$ with $0<\ep_i<1$ satisfying \eqref{eq: chain}. Let $\bve'=(\ep_i')_{i\in I}$  be as in \eqref{eq: ep' new} and $\bx_0=(x_{0,i})_{i\in I}\in \bU$. Suppose that 
\begin{equation}\label{eq5.14q}
\bx\in \Delta=\Delta(\bx_0):=\prod_{i\in I} B\left(x_{0,i}, \left(\frac{\ep_i'}{e^t}\right)^{1/2}\right)\subset \bU
\end{equation}
and $(\bp,q)\in\Z^{n+1}$ satisfy 
\begin{equation}\label{eq5.14}
\begin{aligned}
    &\left\vert g_{j}(\bx)-\frac{p_{j}}{q}\right\vert <\frac{\ep_{j}}{e^t}\qquad (j\in J),\\
    & \left\vert x_i-\frac{p_i}{q}\right\vert<\frac{\ep_i}{e^t}\qquad (i\in I),\\
    & 1\leq q\leq e^t.
    \end{aligned}
\end{equation}
Then 
$$\begin{aligned}
&\left\vert qg_{j}(\bx_0)-\sum_{i\in I(k)} \partial_i g_{j}(\bx_0)(qx_{0,i}-p_i)-p_{j}\right\vert \ll_{F, n}\ep_{j}\qquad (j\in J(k),~1\leq k\leq \fs),\\
    & \left\vert qx_{0,i}-p_i\right\vert<2(\ep_i' e^t)^{1/2} \qquad (i\in I),\\[1ex]
    & 1\leq q\leq e^t.
    \end{aligned}$$
\end{lemma}

\medskip

Define the following generalisation of \eqref{eq4.2} 
$$
\begin{aligned}
   & \mathfrak{a}=\diag(\underbrace{1}_{\#J(\fs)}, \underbrace{\frac{\ep_i}{(\ep_i'e^t)^{1/2}}}_{i\in I(\fs)^\star},\cdots, \underbrace{1}_{\#J(1)},\underbrace{\frac{\ep_i}{(\ep_i'e^t)^{1/2}}}_{i\in I(1)^\star}, 1),
\end{aligned}
$$
where by $I(k)^\star$ we mean the longest permutation of $I(k), 1\le k\le \fs.$
Similar to \S \ref{sec: convergence}, define 
$$\mathfrak{M}=\mathfrak{M}(\bve, t):=\{\bx\in \bU~|~\lambda_{n+1}(\mathfrak{a}g_{conv} U_1(\bx)\Z^{n+1})>\phi\},$$
where $g_{conv}$ is defined in \S \ref{sec: convergence} by \eqref{eq4.1}. 
Again, by the choice of $\ep_i'$, following  the same proof as in Lemma \ref{lem: counting locally on good} we obtain the following

\begin{lemma}\label{lem: gen counting locally on good}
    Let $B$ be a ball in $\bU$. Suppose that \eqref{eq: chain} holds, $t\in\N$, and $(\ep_i')_{i\in I}$  be as in \eqref{eq: ep' new}. For every $\bx_0$ in $B\cap \bU\setminus \mathfrak{M}$, we have that
    $$
    N\left(\Delta(\bx_0)\cap B, \bve, t\right)\ll_{F,n} e^{t(d+1)}\prod_{j\in J}\ep_{j} \left(\frac{\prod_{i\in I}\ep_i'}{e^{dt}}\right)^{1/2}.
    $$
\end{lemma}
Using the above with the covering argument, we get the following lemma generalising  Proposition \ref{prop: count on major}.
\begin{proposition}\label{prop: gen count on major}
Let $B$ be a ball in $V\subset \bU$. For sufficiently large  $t\in\N$ and every  $\bve=(\ep_i)_{i=1}^n$ satisfying \eqref{eq: chain},
    $$N(B\setminus \mathfrak{M},\bve, t)\leq e^{t(d+1)}\prod_{j\in J} \ep_{j}.$$
\end{proposition}

Next, following the same calculation as in \S \ref{sec: minor is small}, we get 
 $$\mathfrak{M}\subset \mathfrak{S}_{F}(\delta,K\dots,K,\bT)$$ with the following choices of parameters:
 $$\begin{aligned}
&\delta=e^{-t},\\ 
& K=\max_{i\in I} K_i, K_i= (\ep'_i e^t)^{-1/2}, ~i\in I, \\ 
& T_i=K_i+(\min_{j\in J(k)}\ep_{j})^{-1}, T_{j}=\ep_{j}^{-1}, i\in I(k), j\in J(k), 1\le k\le \fs.\end{aligned}$$
Note that while computing $T_i, i\in I(k),$ only $(\min_{j\in J(k)}\ep_{j})^{-1}$ shows up instead of $(\min_{j\in J}\ep_{j})^{-1}$ (which is bigger) is due to the parametrization form of $F$ in \eqref{eq: Para_new}.
By the choice of $\ep_{i}'$ we get that 
$T_i=(\min_{j\in J(k)}\ep_{j})^{-1}\leq \ep_i^{-1}, i\in I(k), 1\le k\le \fs. $ 

Then, by Theorem~1.4 in \cite{BKM} (or alternatively Theorem~\ref{BKM_new2}), we get

\begin{proposition}\label{prop: gen Bound on minor}
For $\mu_d$-almost every $\bx_0\in \bU$, there exists a ball $B_0$ centered at $\bx_0$, $\alpha>0$ and a constant $t_0$ depending on $n,F, B_0$ such that  for $t\geq t_0$ and any $\bve=(\ep_i)_{i=1}^n$ satisfying \eqref{eq: chain},
    $$
    \mu_d(\mathfrak{M}\cap B_0) \ll_{F, B_0} \left(\frac{1}{e^{t/2}} e^{-t}\ep_{1}^{-1}\cdots\ep_{n}^{-1}\right)^{\alpha}.
    $$
\end{proposition}

Let $S_F(t;\ve_1,\dots,\ve_n)$ be set of $\bx\in\bU$ such that
\begin{equation}\label{eq5.14+}
\begin{aligned}
    &\left\vert F_{i}(\bx)-\frac{p_{i}}{q}\right\vert <\frac{\ep_{i}}{e^t}\qquad (1\le i\le n),\\
    & 1\leq q\leq e^t.
    \end{aligned}
\end{equation}
holds for some $(\bp,q)\in\Z^{n+1}$. Then using Propositions~\ref{prop: gen count on major} and \ref{prop: gen Bound on minor} we get the following 

\begin{proposition}\label{prop:totalbound}
Let $F:\bU\to\R^n$ be nondegenerate. Then for $\mu_d$-almost every $\bx_0\in \bU$, there exists a ball $B_0$ centered at $\bx_0$, $\alpha>0$ and a constant $t_0$ depending on $n,F, B_0$ such that  for $t\geq t_0$ and any $\bve=(\ep_i)_{i=1}^n\in(0,1)^n$ 
    $$
    \mu_d(S_F(t;\ve_1,\dots,\ve_n)\cap B_0) ~~\ll_{F, B_0} ~~e^t\ve_1\cdots\ve_n+\left(\frac{1}{e^{t/2}} e^{-t}\ep_{1}^{-1}\cdots\ep_{n}^{-1}\right)^{\alpha}
    $$
\end{proposition}

\begin{proof}
To begin with assume  the chain assumption \eqref{eq: chain}. Write $F$ in the form \eqref{eq: Para_new}, which we can do locally in a neighborhood of almost every point. We note that this involves making a diffeomorphic  change of variables, which preserves Lebesgue measure up to a positive bounded factor. Working with the new parameterisation, we observe that $S_F(t;\ve_1,\dots,\ve_n)$ becomes the set of $\bx\in\bU$ satisfying \eqref{eq5.14} for some $(\bp,q)\in\Z^{n+1}$. Then we fall into the framework of Propositions~\ref{prop: gen count on major} and \ref{prop: gen Bound on minor}. Now the first term in the estimate follows from Propositions~\ref{prop: gen count on major} and the second term of the estimate follows from Proposition~\ref{prop: gen Bound on minor}.

Now suppose that  \eqref{eq: chain} does not hold. Then there is a permutation $\omega=(k_1,\cdots, k_n)$ of $(1,\dots,n)$ such that $\omega(\ve_1,\dots,\ve_n)=(\ep_{k_1}, \ldots,\ep_{k_n})$ satisfies \eqref{eq: chain}, that is $\ep_{k_1}\leq \dots\leq\ep_{k_n}$. Clearly, $\omega F$ is nondegenerate since $F$ is nondegenerate. Hence the above argument is applicable to $\omega F$ and $\omega(\ve_1,\dots,\ve_n)$ to give the same estimate, since the set $S_F(t;\ve_1,\dots,\ve_n)$ is invariant under applying any permutation to $F$ and $(\ve_1,\dots,\ve_n)$ simultaneously. This completes the proof.
\end{proof}

The proof of convergence in Theorem~\ref{thm: dream} is now completed as in the special case on using exactly the same argument.

\section{Multiplicative Convergence}\label{sec: multi}

To begin with, we note that, by Proposition~\ref{prop2.1} (with $n=1$), we can assume without loss of generality, that for any $\mathfrak{c}>0$ 
\begin{equation}\label{eq: lower bound on psi}  \psi(q)\ge q^{-1-\mathfrak{c}}\qquad\text{for all $q\in\N$}.\end{equation} 
We fix any $0<\mathfrak{c}<1/2$. Note that the convergence sum condition of Theorem~\ref{thm:mult} is equivalent to
\begin{equation}\label{eq6.2x}
\sum_{t\geq 1} e^t\psi(e^t) t^{n-1}<\infty\,.
\end{equation}
Then, by the monotonicity of $\psi$, we have that
\begin{equation}\label{eq5.2}
\psi(e^t)<e^{-t}
\end{equation}
for all sufficiently large $t$. Without loss of generality, we will assume 
\eqref{eq5.2} for all $t\ge1$. 

Let $F:\bU\to\R^n$ be nondegenerate. For each $t\in\N$ and $w_0>0$ let us define the set
$\tilde S_F(t,w_0)$
to consist of all $\bx\in\bU$ such that 
$$
\min_{1\le i\le n}|qy_i-p_i|< e^{-tw_0}\qquad\text{for  some integer $1\le q\le e^t$ and $\vv p\in\Z^n$}\,.
$$
Observe that 
$$
\limsup_{t\to\infty}\tilde S_F(t;w_0)\subset F^{-1}\cS_n^\times(q^{-w_0}).
$$
It is easily verified that there exists $w_0>1$ such that 
$\mu_d(F^{-1}\cS_n^\times(q^{-w_0}))=0$. 
Indeed, by the theorem of Kleinbock and Margulis \cite{KM} we can take any $w_0>1$, although the existence of a $w_0>1$ can be established by a relatively straightforward use of the Borel-Cantelli lemma. Fix such a $w_0$. Therefore we have that
\begin{equation}\label{eq6.3}
\mu_d(\limsup_{t\to\infty}\tilde S_F(t;w_0))=0.
\end{equation}

Given a vector $\vv k=(k_1,\dots,k_{n-1})\in \Z^{n-1}$ and $t\ge1$, we define
$$
\ve_i(t,\vv k):=\left\{\begin{array}{cl}
    e^{-k_i} & \text{if $1\le i\le n-1$}\,, \\
    \psi(e^t)\prod_{i=1}^{n-1} e^{k_i+1} & \text{if $i=n$}\,. 
\end{array}\right.
$$
Thus,
\begin{equation}\label{eq: prod of vei}
\prod_{i=1}^n\ve_i(t,\vv k)=e^{n-1}\psi(e^t)\asymp \psi(e^t)
\end{equation}
regardless of $\vv k$.
For each $t\ge1$ let 
$$
\mathcal{K}_t:=\{\vv k=(k_1,\dots,k_{n-1}\in \Z^{n-1}~\vert~ 0\le k_i\le w_0t\text{ for each $i$} \}.
$$
Naturally, $\#\mathcal{K}_t\asymp t^{n-1}$.

Then, it is readily seen that
$$
F^{-1}\cSM_n(\psi)\subset \limsup_{t\to\infty}\tilde S_F(t;w_0) \cup \limsup_{t\to\infty}\bigcup_{\vv k\in\mathcal{K}_t}S_F(t;\ve_1(t,\vv k),\dots,\ve_n(t,\vv k)),
$$
where $S_F(t;\ve_1(t,\vv k),\dots,\ve_n(t,\vv k))$ is defined by \eqref{eq5.14+}. By \eqref{eq6.3}, it suffices to prove that the second limsup set is null. Clearly, it is enough to prove this statement on a sufficiently small neighborhood of almost every point $\bx_0$ in $\bU$. Let $\bx_0$, $B_0$ and $\alpha>0$ be as in Proposition~\ref{prop:totalbound}. Then, using \eqref{eq: prod of vei} and \eqref{eq: lower bound on psi}, for any $\vv k\in\mathcal{K}_t$, we get that
$$
\mu_d(S_F(t;\ve_1(t,\vv k),\dots,\ve_n(t,\vv k))\cap B_0) ~~\ll_{F, B_0} ~~e^t\psi(e^t)+e^{-t\alpha(1/2-\mathfrak{c})}.
$$
Since $\mathfrak{c}<1/2$, by \eqref{eq6.2x} and $\#\mathcal{K}_t\asymp t^{n-1}$, we get that  
$$
\sum_t \mu_d\left(\bigcup_{\vv k\in\mathcal{K}_t}S_F(t;\ve_1(t,\vv k),\dots,\ve_n(t,\vv k))\cap B_0)\right)\ll \sum_{t\ge1} e^t\psi(e^t)t^{n-1}+\sum_{t\ge1}e^{-t\alpha(1/2-\mathfrak{c})}t^{n-1}<\infty
$$
and the Borel-Cantelli Lemma completes the proof.

\bigskip

\noindent\emph{Acknowledgments.} VB and SD were supported by an EPSRC grant EP/Y016769/1. LY is supported in part by a start-up grant from National University of Singapore. SD thanks her colleagues at the Department of York, where majority of this work was done. SD thanks Han and Sam for informing her about their ongoing work \cite{MultCSTY} with Niclas and Rajula during her visit to the University of Warwick, and also for their warm hospitality. We thank all authors of \cite{MultCSTY} for motivating us to finish our paper. Last but not the least, SD thanks Subhajit for tolerating $p/q$'s everywhere and every single day at York.

\bibliographystyle{abbrv}
\bibliography{inhomosim.bib}

\end{document}